\documentclass[12pt]{amsart}

\usepackage[vcentermath,enableskew]{youngtab}
\usepackage{amssymb}
\usepackage[all]{xy}

\newcommand{\version}{Ver.~0.0}
\newcommand{\setversion}[1]{\renewcommand{\version}{Ver.~{#1}}}
\setversion{0.0 [2014/06/09]}
\setversion{0.1 [2014/07/04 00:29:35]}
\setversion{1.0 [2016/12/26]}
\setversion{1.1 [2016/02/03 00:05:47]}
\setversion{2.0 [2016/03/18]}

\newtheorem{theorem}{Theorem}[section]
\newtheorem{lemma}[theorem]{Lemma}
\newtheorem{proposition}[theorem]{Proposition}
\newtheorem{corollary}[theorem]{Corollary}
\theoremstyle{definition}
\newtheorem{example}[theorem]{Example}
\newtheorem{notation}[theorem]{Notation}
\newtheorem{definition}[theorem]{Definition}
\newtheorem{remark}[theorem]{Remark}

\newcommand{\Xfv}{\mathfrak{X}}
\newcommand{\Gr}{\mathrm{Gr}}
\newcommand{\Lie}{\mathrm{Lie}}
\newcommand{\GL}{\mathrm{GL}}
\newcommand{\Sp}{\mathrm{Sp}}
\newcommand{\rk}{\mathrm{rk}\,}
\newcommand{\smallyoung}[1]{\mbox{\footnotesize $\young(#1)$}}
\newcommand{\verysmallyoung}[1]{\mbox{\tiny $\young(#1)$}}
\newcommand{\pbar}{\,{+}\,\rule[-2pt]{1.4pt}{2.55ex}}
\newcommand{\barm}{\hspace*{-.1pt}\rule[-2pt]{1.4pt}{2.55ex}\hspace*{.8pt}{-}\,}
\newcommand{\spbar}{\,{+}\hspace*{.6pt}\rule[-2pt]{1.7pt}{2.65ex}}
\newcommand{\sbarm}{\hspace*{-.1pt}\rule[-2pt]{1.4pt}{2.65ex}\hspace*{.8pt}{-}\,}

\newcommand{\mattwo}[4]{\Bigl(\begin{array}{@{\,}c@{\;\;}c@{\,}}{#1} & {#2} \\ {#3} & {#4} \end{array}\Bigr)}

\newcommand{\smallmatrixt}[1]{\renewcommand{\arraystretch}{0.4}
\setlength{\arraycolsep}{1.5pt}
\left(\begin{array}{ccc|c} #1
\end{array}\right)}

\begin{document}

\title{On the exotic Grassmannian and its nilpotent variety}
\author{Lucas Fresse}\thanks{L. Fresse is supported by the ISF Grant Nr. 882/10 and by the ANR project NilpOrbRT
(ANR-12-PDOC-0031).}
\address{Universit\'e de Lorraine, CNRS, Institut \'Elie Cartan de Lorraine, UMR 7502, Vandoeu\-vre-l\`es-Nancy, F-54506, France}
\email{lucas.fresse@univ-lorraine.fr}
\author{Kyo Nishiyama}\thanks{K. Nishiyama is supported by the JSPS Grant-in-Aid for Scientific Research 
\#{25610008}}
\address{Department of Physics and Mathematics, Aoyama Gakuin University, Fuchinobe 5-10-1, Sagamihara 252-5258, Japan}
\email{kyo@gem.aoyama.ac.jp}

\begin{abstract}
Given a decomposition
of a vector space $V=V_1\oplus V_2$, the direct product $\mathfrak{X}$ of the projective space $\mathbb{P}(V_1)$ with a Grassmann variety $\mathrm{Gr}_k(V)$
can be viewed as a double flag variety for the symmetric pair $(G,K)=(\mathrm{GL}(V),\mathrm{GL}(V_1)\times\mathrm{GL}(V_2))$.
Relying on the conormal variety for the action of $K$ on $\mathfrak{X}$,
we show a geometric correspondence between the $K$-orbits of $\mathfrak{X}$ and
the $K$-orbits of some appropriate exotic nilpotent cone.
We also give a combinatorial interpretation of this correspondence in some special cases.
Our construction is inspired by a classical result of Steinberg \cite{Steinberg.1976}
and by the recent work of Henderson and Trapa \cite{Henderson.Trapa.2012} for the symmetric pair $(\mathrm{GL}(V),\mathrm{Sp}(V))$.
\end{abstract}

\maketitle

\section{Introduction}

\subsection{Multiple flag varieties} Let $G$ be a connected reductive algebraic group over $\mathbb{C}$ and let $\mathfrak{g}$ be its Lie algebra. Let $(g,x)\mapsto g\cdot x$ denote the adjoint action of $G$ on $\mathfrak{g}$.

Given a parabolic subgroup $P\subset G$, the quotient $G/P$ is a projective variety called (partial) flag variety. 
Equivalently $G/P$ can be viewed as the set of parabolic subgroups of the same type as $P$ (i.e., conjugate to $P$), or as the set of parabolic subalgebras $\mathfrak{p}_1\subset\mathfrak{g}$ of the same type as $\mathfrak{p}:=\Lie(P)$ (i.e., of the form $g\cdot\mathfrak{p}$ for $g\in G$). 
Evidently the natural action of $G$ on $G/P$ is transitive.

Flag varieties are central objects in geometric representation theory.
In recent years there has been a growing interest in multiple flag varieties, or, 
direct products of flag varieties. 
By the Bruhat decomposition, a double flag variety of the form $G/P_1\times G/P_2$ always consists of finitely many orbits for the diagonal action of $G$, which are parametrized by the double cosets $W_{P_1}wW_{P_2}$ ($w\in W$) 
of the Weyl group $ W = W_G$.
Triple flag varieties of the form $G/P_1\times G/P_2\times G/P_3$ consist of 
infinitely many $G$-orbits in general. 
Triple flag varieties with finite number of $G$-orbits are classified in  
\cite{MWZ.1999}, \cite{MWZ.2000} 
in the classical cases (see also \cite{Matsuki.2013}).

In this paper we study certain multiple flag varieties that can be associated to symmetric pairs.

\subsection{Double flag variety of a symmetric pair} Let $\theta$ be an involutive automorphism of $G$. The subgroup $K=G^\theta:=\{g\in G:\theta(g)=g\}$ is then reductive, and the pair $(G,K)$ is called a {\em symmetric pair}. 

\begin{example}
{\rm (a)} For $\theta=\mathrm{id}_G$ we obtain the (trivial) symmetric pair $(G,G)$. \\
{\rm (b)} Let $G=\GL_{2n}(\mathbb{C})$ and let $\theta(g)=\iota {}^tg^{-1}\iota^{-1}$, 
where $\iota=\left(\begin{smallmatrix}
0 & 1_n \\ -1_n & 0
\end{smallmatrix}\right)$. Then $G^\theta=\Sp_{2n}(\mathbb{C})$. The pair $(\GL_{2n}(\mathbb{C}),\Sp_{2n}(\mathbb{C}))$ is referred to as a symmetric pair of type AII. \\
{\rm (c)} Similarly,  $(G,K)=(\GL_{p+q}(\mathbb{C}),\GL_p(\mathbb{C})\times \GL_q(\mathbb{C}))$ is a symmetric pair, which is called of type AIII. 
In this case, the involution $ \theta $ is given by an inner conjugation by $\left(\begin{smallmatrix}
1_p & 0 \\ 0 & -1_q
\end{smallmatrix}\right)$.  

We refer the readers to 
\cite{Helgason.DG.1978} 
for a complete classification of classical symmetric pairs.
\end{example}

Following \cite{NO.2011}, for parabolic subgroups $Q\subset K$ and $P\subset G$, we consider the double flag variety
\[\Xfv:=G/P\times K/Q.\]
The group $K$ acts diagonally on $\Xfv$ and we call $\Xfv$ of {\em finite type} if it has a finite number of $K$-orbits.
This is not the case in general, as shown by the next examples.

\begin{example}
\label{E1.2}
{\rm (a)} In the case of the trivial symmetric pair $(G,G)$, the double flag variety $\Xfv$ is a usual double flag variety for $G$ of the form $G/P\times G/Q$, and it is of finite type for every choice of $Q$ and $P$ 
(by the Bruhat decomposition as already explained above). 
\\
{\rm (b)} Consider $\theta:G\times G\to G\times G$ given by the flip $\theta(g_1,g_2)=(g_2,g_1)$. 
Then $G^\theta=\{(g,g):g\in G\}\cong G$. 
In this case $\Xfv$ is a triple flag variety of the form 
$(G\times G)/(P_1\times P_2)\times G/Q=G/P_1\times G/P_2\times G/Q$, which is in general not of finite type 
(see \cite{MWZ.1999} and \cite{MWZ.2000} as already mentioned).
\end{example}

Sufficient criteria for $\Xfv$ to be of finite type are obtained in \cite{HNOO.2013,NO.2011}. 
Moreover, in \cite{HNOO.2013}, the complete classification of the double flag varieties 
$ \Xfv = G/P \times K/Q $ of finite type is given when $ P $ or $ Q $ is a Borel subgroup of $ G $ or $ K $, respectively.

In the case of the symmetric pair of type AIII, we know:

\begin{proposition}[{see \cite[Table 3]{NO.2011}}]
\label{proposition-1}
Let $(G,K)=(\GL_{p+q}(\mathbb{C}),\GL_p(\mathbb{C})\times \GL_q(\mathbb{C}))$.
In each of the following cases,
the variety $\Xfv=G/P\times K/Q$ is of finite type. 
\begin{itemize}
\item[\rm (a)] $Q\subset K$ is mirabolic {\upshape(}i.e., $K/Q$ is a projective space{\upshape)} and $P$ is arbitrary;
\item[\rm (b)] $P\subset G$ is maximal 
{\upshape(}i.e., $G/P$ is a Grassmannian{\upshape)} and $Q$ is arbitrary.
\end{itemize}
\end{proposition}

\begin{remark}
In \cite[Table 3]{NO.2011}, more examples of $ \Xfv $ of type AIII which are of finite type are given.  
But, even in the case of type AIII, the classification of double flag varieties of finite type is open.
\end{remark}

\begin{notation}
\label{notation-1-5new}
Hereafter $(G, K)$ denotes a symmetric pair.  
By differentiation, 
the automorphism $\theta$ induces a Lie algebra automorphism still denoted by the same letter 
$\theta\in\mathrm{Aut}(\mathfrak{g})$. 
Let 
\begin{equation*}
\mathfrak{g}=\mathfrak{k}\oplus\mathfrak{s}, 
\quad \text{ where } \quad 
\mathfrak{k}=\Lie(K)=\ker(\theta-\mathrm{id}_\mathfrak{g})
\; \text{ and } \;  
\mathfrak{s}=\ker(\theta+\mathrm{id}_\mathfrak{g}). 
\end{equation*}
For every $x\in\mathfrak{g}$, we write
$x=x^\theta+x^{-\theta}$ with $(x^\theta,x^{-\theta})\in\mathfrak{k}\times\mathfrak{s}$, or to be more explicit, 
\begin{equation*}
x^{\theta}=\frac{1}{2}(x+\theta(x))\qquad\mbox{and}\qquad x^{-\theta}=\frac{1}{2}(x-\theta(x)).
\end{equation*}
Finally, we denote by $\mathcal{N}(\mathfrak{g})$, $\mathcal{N}(\mathfrak{k})$, and $\mathcal{N}(\mathfrak{s})$ the sets of nilpotent elements of $\mathfrak{g}$, $\mathfrak{k}$, and $\mathfrak{s}$, respectively.
\end{notation}

\subsection{Conormal variety}\label{section-1-3}

The action of the group $K$ on the double flag variety $\Xfv=K/Q\times G/P$ 
induces a Hamiltonian action of $K$ on the cotangent bundle $T^*\Xfv$, 
which therefore gives rise to a moment map $\mu_\Xfv:T^*\Xfv\to\mathfrak{k}^*\cong\mathfrak{k}$ 
(see \cite{Chriss.Ginzburg.1997} for example).  

Let us describe the moment map $ \mu_\Xfv $ explicitly.  
For that purpose, we realize the cotangent bundle over $ G/P $ as 
\begin{equation*}
T^*(G/P) = \{ (\mathfrak{p}_1,x) : \mathfrak{p}_1 \text{ is $ G $-conjugate to } \mathfrak{p}, \; 
x \in \mathfrak{n}_{\mathfrak{p}_1} \} \simeq G \times_P \mathfrak{n}_{\mathfrak{p}} , 
\end{equation*}
where $\mathfrak{n}_{\mathfrak{p}_1}$ stands for the nilpotent radical of 
a parabolic subalgebra $\mathfrak{p}_1\subset\mathfrak{g}$. 
Similarly we realize $ T^\ast(K/Q) $, and then get 
the cotangent bundle $T^*\Xfv=T^*(G/P)\times T^*(K/Q)$ as the set of quadruples
\[
T^*\Xfv = 
\{ ((\mathfrak{p}_1,x), (\mathfrak{q}_1,y)) \in (G/P\times\mathfrak{g}) \times (K/Q\times\mathfrak{k}): 
x \in \mathfrak{n}_{\mathfrak{p}_1},\ y\in\mathfrak{n}_{\mathfrak{q}_1}\} . 
\]
Here $\mathfrak{n}_{\mathfrak{q}_1}$ denotes the nilpotent radical of 
a parabolic subalgebra $\mathfrak{q}_1\subset\mathfrak{k}$, 
and $ \mathfrak{q}_1 $ is a $ K $-conjugate of $ \mathfrak{q} $.  
Then the moment map $\mu_\Xfv$ can be viewed as the map
\[
\mu_\Xfv : T^*\Xfv \to \mathfrak{k}, \qquad \mu_\Xfv((\mathfrak{p}_1,x),(\mathfrak{q}_1,y)) = x^\theta + y .
\]
%
The null fiber $\mathcal{Y}:=\mu_\Xfv^{-1}(\{0\})\subset T^*\Xfv$
of $\mu_\Xfv$ is called {\it conormal variety}. 
It is a closed subvariety in the cotangent bundle and 
we may identify it as 
\begin{equation}
\label{eq:conormal-variety-Y}
\mathcal{Y}=\{(\mathfrak{p}_1,\mathfrak{q}_1,x)\in G/P\times K/Q\times\mathfrak{g}:x\in\mathfrak{n}_{\mathfrak{p}_1},\ x^\theta\in\mathfrak{n}_{\mathfrak{q}_1}\}, 
\end{equation}
since $ y = - x^{\theta} $ is determined by $ x $.

Every $K$-orbit $\mathbb{O}\subset\Xfv$ gives rise to a conormal bundle
$T^*_\mathbb{O}\Xfv:=\bigcup_{z\in\mathbb{O}}(T_z\mathbb{O})^\perp$, which is a Lagrangian, 
smooth, irreducible subvariety of $T^*\Xfv$ of dimension $\dim\Xfv$.  
From the definition of $\mathcal{Y}$ it is readily seen that
the inclusion $T^*_\mathbb{O}\Xfv\subset\mathcal{Y}$ holds.  
In fact, in the case where $\Xfv$ is of finite type, we obtain a finite decomposition
\begin{equation}
\label{1.1}
\mathcal{Y}=\bigsqcup_{\mathbb{O}\in\Xfv/K}T^*_\mathbb{O}\Xfv=\bigcup_{\mathbb{O}\in\Xfv/K}\overline{T^*_\mathbb{O}\Xfv}
\end{equation}
into closed subvarieties of the same dimension.  
Thus the conormal variety $\mathcal{Y}$ is equidimensional and its irreducible components are exactly the closures 
$\overline{T^*_\mathbb{O}\Xfv}$ for $\mathbb{O}\in\Xfv/K$.
They are parametrized by the orbit set $\Xfv/K$.

So if we can parametrize the irreducible components of $ \mathcal{Y} $ nicely by 
different objects, then we can describe the $ K $-orbits in $ \Xfv $ using such parametrization.  
In fact, Steinberg did this for the trivial symmetric pair in his classical paper \cite{Steinberg.1976}.  Let us explain it as a guiding principle.



\subsection{Steinberg correspondence}

We can re-interpret results of Steinberg \cite{Steinberg.1976} using the above mentioned idea 
in the following way.  
This corresponds to the case of the trivial symmetric pair $ (G, K) = (G, G) $, 
where $ \theta $ is considered to be the identity. 

Let $P=Q=B$ be a Borel subgroup of $G$. 
On the one hand, by Bruhat decomposition, 
the $G$-orbits of the double flag variety $\Xfv=G/B\times G/B$ are parametrized by the Weyl group elements $w\in W$,
hence so are the components of $\mathcal{Y}$.
On the other hand, in the present situation, the conormal variety can be described as
\[\mathcal{Y}=\{(\mathfrak{b}_1,\mathfrak{b}_2,x)\in G/B\times G/B\times\mathfrak{g}:x\in\mathfrak{n}_{\mathfrak{b}_1}\cap\mathfrak{n}_{\mathfrak{b}_2}\},\]
and we can consider the map
\[
\pi : \mathcal{Y} \to \mathcal{N}(\mathfrak{g}),
\qquad 
(\mathfrak{b}_1,\mathfrak{b}_2,x) \mapsto x.
\]
The nilpotent cone $\mathcal{N}=\mathcal{N}(\mathfrak{g})$ consists of finitely many $G$-orbits.
The map $\pi$ is a fibration over each orbit $Gx$, 
with fiber $\mathcal{B}_x\times\mathcal{B}_x$, 
where 
\begin{equation*}
\mathcal{B}_x=\{\mathfrak{b}_1\in G/B:x\in\mathfrak{n}_{\mathfrak{b}_1}\}
\end{equation*}
is the Springer fiber of $x$ (an equidimensional variety).
Moreover it can be shown that $\dim\pi^{-1}(Gx)$ does not depend on $x\in\mathcal{N}(\mathfrak{g})$.
Altogether this yields (explicit) bijections
\begin{equation}
\label{1.2}
W\cong \mathrm{Irr}(\mathcal{Y})\cong\bigsqcup_{x\in\mathcal{N}/G}\mathrm{Irr}(\pi^{-1}(Gx))
\cong\bigsqcup_{x\in\mathcal{N}/G}(\mathrm{Irr}(\mathcal{B}_x)\times\mathrm{Irr}(\mathcal{B}_x))/Z_G(x),
\end{equation}
where $\mathrm{Irr}(Z)$ stands for the set of irreducible components of a variety $Z$.
In the right-hand side of (\ref{1.2})
the quotient by the diagonal action of the stabilizer
$Z_G(x):=\{g\in G:g\cdot x=x\}$ on pairs of components of $\mathcal{B}_x$ is considered.
We refer to \cite{Steinberg.1976} for more details.

When $G=\GL_n(\mathbb{C})$, the Weyl group $W=\mathfrak{S}_n$ is the symmetric group, 
the nilpotent orbits $Gx\subset\mathcal{N}$ are parametrized by partitions $\lambda\vdash n$,
the set $\mathrm{Irr}(\mathcal{B}_x)$ is in bijection with the set $\mathrm{STab}(\lambda)$ of standard Young tableaux of shape $\lambda$ (see \cite[\S II.5]{Spaltenstein.1982} and \cite[\S 5]{Steinberg.1976}),
and $Z_G(x)$ is connected hence acts trivially on $\mathrm{Irr}(\mathcal{B}_x)$.
Thus (\ref{1.2}) yields a bijection
\begin{equation}
\mathfrak{S}_n\cong\bigsqcup_{\lambda\vdash n}\mathrm{STab}(\lambda)\times \mathrm{STab}(\lambda),
\end{equation}
which actually coincides with the classical {\it Robinson--Schensted correspondence}
\cite{Steinberg.1988}.

The Steinberg correspondence is closely related to the Springer correspondence \cite{Springer.1976} between nilpotent orbits and Weyl group representations, in the sense that (\ref{1.2}) yields a decomposition of $W$ into cells, which is a basic ingredient for obtaining further realizations and interpretations of Springer representations.
The notion of cells in Weyl groups is originally developed by Joseph \cite{Joseph.1978, Joseph.1979} and later on by Kazhdan and Lusztig \cite{Kazhdan.Lusztig.1979}. In the case of $G=\GL_n(\mathbb{C})$ the Springer correspondence is a direct bijection between the set of nilpotent orbits $\mathcal{N}/G$ and the set $\widehat{W}$ of irreducible representations of $W=\mathfrak{S}_n$. 
Beyond the case of $\GL_n(\mathbb{C})$ the correspondence is more complicated, 
mainly due to the topology of nilpotent $G$-orbits (e.g., the non-connectedness of the stabilizer $Z_G(x)$).

The last fact motivates the search of alternative geometric constructions of Weyl group representations 
in classical cases other than type A.
By considering the symmetric pair of type AII $(G,K)=(\GL_{2n}(\mathbb{C}),\Sp_{2n}(\mathbb{C}))$, 
and relying on properties of affine Hecke algebras,
Kato \cite{Kato.2009, Kato.2011} establishes a bijection ({\it exotic Springer correspondence}) 
between the $\Sp_{2n}(\mathbb{C})$-orbits of the {\it exotic nilpotent cone} 
$\mathbb{C}^{2n}\times\mathcal{N}(\mathfrak{s})$ 
and the irreducible representations of the Weyl group of $\Sp_{2n}(\mathbb{C})$. 
Shoji and Sorlin \cite{Shoji.Sorlin.2013} retrieve this correspondence via character sheaves on 
the exotic symmetric space $\mathbb{C}^{2n}\times \bigl(\GL_{2n}(\mathbb{C})/\Sp_{2n}(\mathbb{C})\bigr) $.

Finally Henderson and Trapa \cite{Henderson.Trapa.2012} interpret Kato's correspondence 
in terms of a relation (an exotic version of Steinberg correspondence) 
between $\Sp_{2n}(\mathbb{C})$-orbits in the exotic nilpotent cone $\mathbb{C}^{2n}\times\mathcal{N}(\mathfrak{s})$ 
and $\Sp_{2n}(\mathbb{C})$-orbits in the variety $\mathbb{C}^{2n}\times \bigl( \GL_{2n}(\mathbb{C})/B \bigr) $ 
(analogous to the double flag variety 
\begin{equation*}
\mathbb{P}(\mathbb{C}^{2n})\times \bigl( \GL_{2n}(\mathbb{C})/B \bigr) = K/Q\times G/B
\end{equation*}
with $Q\subset K$ mirabolic).
Our work is inspired by 
\cite{Henderson.Trapa.2012}.

\subsection{Exotic Grassmannian and exotic nilpotent variety}\label{section-1-5}

Our goal in this paper is to establish an exotic version of Steinberg correspondence for a symmetric pair of type AIII.
Specifically
we consider the following double flag variety.

\begin{notation}\label{notation-1-5}
We use the following notation throughout the paper.

{\rm (a)} We let $V=\mathbb{C}^n$ ($n\geq 1$) and consider the decomposition
\[V=V_1\oplus V_2\ \ \mbox{with}\ \ V_1=\mathbb{C}^p\times\{0\}^{q},\ V_2=\{0\}^{p}\times\mathbb{C}^q\quad(p+q=n).\]
Let $(G,K)$ be the symmetric pair given by
\begin{equation*}
G=\GL(V), 
\qquad
K=\GL(V_1)\times \GL(V_2)\subset G .
\end{equation*}
Let $k\in\{0,\ldots,n\}$. By $\Gr_k(V)$  we denote the Grassmann variety of $k$-dimensional subspaces of $V$.
The {\em exotic Grassmannian} is the double flag variety 
\begin{equation}\label{eq:definition-of-exotic-Grassmannian}
\Xfv:=\Gr_k(V)\times \mathbb{P}(V_1).
\end{equation}
Thus $\Xfv$ is of the form $G/P\times K/Q$, where 
$P\subset G$ is a maximal parabolic subgroup stabilizing a $ k $-space in $ V $ and 
$Q\subset K$ is a mirabolic subgroup stabilizing a line in $ V_1 $.  
In fact, we see that $Q=Q_1\times \GL(V_2)$ with $Q_1\subset \GL(V_1)$ mirabolic. 

{\rm (b)}
Let $ \mathfrak{g}:=\Lie(G)=\mathcal{L}(V) $, and set 
\begin{equation*}
\mathfrak{k} := \Lie(K) = \{
\mattwo{x_1}{0}{0}{x_2} 
: x_i\in\mathcal{L}(V_i) \},
\quad
\mathfrak{s} := \{
\mattwo{0}{x_{12}}{x_{21}}{0}
: x_{ij} \in \mathcal{L}(V_j,V_i) \} . 
\end{equation*}
For 
$ x = \mattwo{x_1}{x_{12}}{x_{21}}{x_2}
\in \mathfrak{g} $ we denote 
$ x^\theta = 
\mattwo{x_1}{0}{0}{x_2} 
\in \mathfrak{k} $ and 
$ x^{-\theta} = 
\mattwo{0}{x_{12}}{x_{21}}{0}
\in \mathfrak{s} $.
The conormal variety $\mathcal{Y}$ corresponding to the $K$-variety $\Xfv$ can be described as
\[\mathcal{Y}=\{(W,L,x)\in \Gr_k(V)\times\mathbb{P}(V_1)\times\mathcal{L}(V):\mathrm{Im}\,x\subset W\subset\ker x,\ \mathrm{Im}\,x^\theta\subset L\subset\ker x^\theta\}.\]
This is an explicit realization of $ \mathcal{Y} $ in this case, 
the general definition being given in \eqref{eq:conormal-variety-Y}.

{\rm (c)} 
Let us define the ($K$-equivariant) map 
\[\pi:\mathcal{Y} \to \mathbb{P}(V_1) \times \mathfrak{s}, \qquad (W,L,x)\mapsto (L,x^{-\theta}). \]
Then, as we will prove in Section \ref{section-3}, $ x^{-\theta} $ is nilpotent and belongs to 
\begin{equation*}
\mathcal{N}(\mathfrak{s})_2^k := \{ x\in\mathfrak{s} : x^2 = 0 \ \text{ and } \, \mathrm{rk}\, x \leq \min\{k,n-k\} \} ;
\end{equation*}
in addition the image $ \pi(\mathcal{Y}) $ coincides with 
\[\mathfrak{E}:=\mathbb{P}(V_1)\times\mathcal{N}(\mathfrak{s})_2^k,\]
and we call it {\em exotic nilpotent cone}.  
Thus we have a surjective $ K $-equivariant map $ \pi : \mathcal{Y} \to \mathfrak{E} $ 
(see Proposition~\ref{proposition-2-8}; proved in Section~\ref{subsection:proof-prop-2-8}).
\end{notation}

Our analysis of the exotic Grassmannian $\Xfv$ relies on the following key facts.
\begin{itemize}
\item
The diagonal action of $K$ on $\Xfv$ has a finite number of orbits (this follows from Proposition \ref{proposition-1}; see also Section \ref{section-2-3} for a combinatorial description of the orbits). Therefore $\mathcal{Y}$ is equidimensional and 
$\mathrm{Irr}(\mathcal{Y})\cong \Xfv/K$ (see Section \ref{section-1-3}).
\item
The diagonal action of $K$ on $\mathfrak{E}$ has a finite number of orbits
(this result is due to Johnson \cite{Johnson.2010.MR}; see Section \ref{section-2-1}).
\end{itemize}
Using the map $ \pi : \mathcal{Y} \to \mathfrak{E} $, 
we relate $K$-orbits in $\mathfrak{E}$ and components of $\mathcal{Y}$.
Thus we can describe the relation between $ K $-orbits in $ \Xfv $ and $ \mathfrak{E} $, 
which is our main result.

\begin{theorem}[cf., Theorem~\ref{theorem-2}]\label{theorem-1}
Let $\Xfv,\mathcal{Y},\mathfrak{E},p,q,k$ be as above. 

{\rm (a)} For every component $\mathcal{C}\subset\mathcal{Y}$, there is a unique $K$-orbit
$\mathfrak{O}_\mathcal{C}\subset\mathfrak{E}$ such that $\mathcal{C}\subset \overline{\pi^{-1}(\mathfrak{O}_\mathcal{C})}$. 
Moreover, the so-obtained map $\Xi:\mathrm{Irr}(\mathcal{Y})\to\mathfrak{E}/K$, $\mathcal{C}\mapsto\mathfrak{O}_\mathcal{C}$ is surjective, and each fiber of $\Xi$ 
contains at most two elements.

{\rm (b)} In the case where $p\leq \max\{k,n-k,q+1\}$, the map $\Xi$ is bijective.
\end{theorem}

By Theorem \ref{theorem-1} and Section \ref{section-1-3}, we obtain the following corollary, which can be viewed as an exotic version of Steinberg correspondence for the symmetric pair of type AIII.

\begin{corollary}[cf., Corollary~\ref{corollary-2-11}]\label{corollary-1}
{\rm (a)} There is an explicit surjection $\Phi:\Xfv/K\to\mathfrak{E}/K$, $\mathbb{O}\mapsto
\mathfrak{O}_{\overline{T_\mathbb{O}^*\Xfv}}$, whose fibers have at most two elements.  

{\rm (b)} In the case where $p\leq \max\{k,n-k,q+1\}$, the map $\Phi$ is bijective.
\end{corollary}

Actually 
Theorem~\ref{theorem-1} and Corollary~\ref{corollary-1} can be stated in a more precise way. 
Indeed from Johnson \cite{Johnson.2010.MR} 
we have a parametrization of the $K$-orbits of the exotic nilpotent cone $\mathfrak{E}$ 
by so-called striped signed diagrams; see Section \ref{section-2-1}. 
In Section \ref{section-2-2} we state refined versions of Theorem \ref{theorem-1} and Corollary \ref{corollary-1} (Theorem~\ref{theorem-2} and Corollary \ref{corollary-2-11}), 
which include a characterization of $K$-orbits $\mathfrak{O}\subset\mathfrak{E}$ 
such that $\Xi^{-1}(\mathfrak{O})$ and $\Phi^{-1}(\mathfrak{O})$ are singletons (resp., pairs).

Furthermore in Section \ref{section-2-3} we give a combinatorial parametrization of the $K$-orbits of the double flag variety $\Xfv$ and, for certain values of $p,q,k$, we 
explicitly compute the correspondence $\Phi$ of Corollary \ref{corollary-1}. Although it is a priori computable, the correspondence $\Phi$ appears to be not straightforward in general.


Let us mention relevant works in the literature.  
In \cite{Henderson.Trapa.2012},  Henderson and Trapa consider the symmetric pair $(\GL_{2n}(\mathbb{C}), \Sp_{2n}(\mathbb{C}))$
and the enhanced flag variety 
$\hat{\mathfrak{X}}:=\mathbb{C}^{2n}\times \bigl( \GL_{2n}(\mathbb{C})/B \bigr) $, 
whose projectivization is the double flag variety $\hat{\mathfrak{X}}':=\Sp_{2n}(\mathbb{C})/Q\times\GL_{2n}(\mathbb{C})/B$ for $Q$ mirabolic.  
Their results include a correspondence between $ K $-orbits in $ \hat{\mathfrak{X}} $ and 
$K$-orbits in the exotic nullcone. 
In this respect they use the description of $\hat{\mathfrak{X}}/\Sp_{2n}(\mathbb{C})$ due to Matsuki \cite{Matsuki.2013} and the description of the exotic nilpotent cone of Kato \cite{Kato.2009}. 
They also make use of the enhanced nilpotent cone studied by Achar and Henderson \cite{Achar.Henderson.2008}. 
A significant difference with our situation is that 
the variety $\hat{\mathfrak{X}}$ (as well as $\hat{\mathfrak{X}}'$) is endowed with 
a diagonal action of the full group  $\GL_{2n}(\mathbb{C})$. 
Our approach is therefore quite different from that of \cite{Henderson.Trapa.2012}.

We also mention that there are some related works on triple flag varieties, which are special cases of the double flag varieties of symmetric pairs as seen from Example \ref{E1.2}\,{\rm (b)}.  
Specifically Travkin \cite{Travkin.2009} and Rosso \cite{Rosso.2012} describe geometric correspondences (``mirabolic RSK correspondences'') relating triple flag varieties and enhanced nilpotent varieties.

\section{Precise statement of main result}

In this section we give a detailed formulation of the results outlined in Section \ref{section-1-5}. 
In Section \ref{section-2-1}
we parametrize (following Johnson \cite{Johnson.2010.MR})  the $K$-orbits of the exotic nilpotent cone $\mathfrak{E}$ in the case of the exotic Grassmannian $\Xfv$. In Section \ref{section-2-2} we state our main result on the correspondence between the $K$-orbits of $\mathfrak{E}$ and $\Xfv$ (Theorem \ref{theorem-2}).
In Section \ref{section-2-3}, we parametrize the $K$-orbits of $\Xfv$ by so-called $(1,2)$-tableaux, and we give a combinatorial interpretation of the correspondence of Theorem \ref{theorem-2} in certain particular cases. All results presented in this section are proved in the subsequent sections.

\subsection{Exotic nilpotent cone and its $K$-orbits}

\label{section-2-1}

We consider the notation introduced in Notation \ref{notation-1-5}.
In particular $\mathfrak{s}$ consists of elements of the form 
$ x = 
\mattwo{0}{a}{b}{0}
$ with $a\in\mathcal{L}(V_2,V_1)$ and $b\in\mathcal{L}(V_1,V_2)$, and we have $\rk x=\rk a+\rk b$.
We consider the diagonal action of $K=\GL(V_1)\times\GL(V_2)$ on the product 
$\mathbb{P}(V_1)\times\mathcal{N}(\mathfrak{s})$.

\begin{proposition}[see \cite{Johnson.2010.MR}]
\label{proposition-2-1}
$\mathbb{P}(V_1)\times\mathcal{N}(\mathfrak{s})$ has a finite number of $K$-orbits.
\end{proposition}

Furthermore Johnson's result in \cite{Johnson.2010.MR} contains a combinatorial parametrization of the orbits. 

In this paper, we focus on the action of $K$
on the 2-step nilpotent cone 
$\mathcal{N}(\mathfrak{s})_2 := \{ x = 
\mattwo{0}{a}{b}{0}
\in\mathfrak{s} : x^2 = 0 \} $
and hence on the exotic nilpotent cone $\mathfrak{E}=\mathbb{P}(V_1)\times\mathcal{N}(\mathfrak{s})_2^k$ 
(see Notation \ref{notation-1-5}\,{\rm (c)}).
Note that
\[
\mathcal{N}(\mathfrak{s})_2 = \Big\{
\mattwo{0}{a}{b}{0}
\in\mathfrak{s}:ab=0\quad\mbox{and}\quad ba=0\Big\},\]
whence $\rk a+\rk b\leq \min\{\dim V_1,\dim V_2\}=\min\{p,q\}$
whenever 
$
\mattwo{0}{a}{b}{0}
\in\mathcal{N}(\mathfrak{s})_2 $.

\begin{definition}
\label{definition-2-1}
{\rm (a)} Let $\Lambda_2$ be the set of pairs of nonnegative integers $(r,s)$ such that $r+s\leq\min\{p,q\}$. \\
{\rm (b)} Let $\Pi_2$ be the set of pairs $((r,s),\mu)$ where $(r,s)\in\Lambda_2$ and $\mu=(\mu_1,\ldots,\mu_{n-(r+s)})$ is a sequence satisfying the following conditions:
\begin{equation}
\left\{\begin{array}{cl}
{\rm (a)} & \mu_1=\ldots=\mu_r=1; \\ 
{\rm (b)} & \mu_{r+1}=\ldots=\mu_{r+s}\in\{0,2\}; \\ 
{\rm (c)} & \mu_{r+s+1}=\ldots=\mu_p\in\{-1,1\}; \\
{\rm (d)} & \mu_{p+1}=\ldots=\mu_{n-(r+s)}=0; \\
{\rm (e)} & \mbox{at least one term $\mu_i$ is $\geq 1$ (automatically satisfied if $r\not=0$)}; \\
{\rm (f)} & \{-1,2\}\not\subset\{\mu_1,\ldots,\mu_{n-(r+s)}\}. 
\end{array}\right.
\end{equation}
Furthermore we call $((r,s),\mu)$ of type (I) if $\mu_{r+1},\ldots,\mu_p\leq 0$;
of type (II) if $r+s\not=p$, $\mu_{r+1}=\ldots=\mu_{r+s}=0$, and $\mu_{r+s+1}=\ldots=\mu_p=1$;
of type (III) if $s\not=0$ and $\mu_{r+1}=\ldots=\mu_{r+s}=2$.

\end{definition}



\begin{notation}
{\rm (a)} It is convenient to represent a pair $(r,s)\in\Lambda_2$ by a {\em signed Young diagram $\tilde\lambda_{(r,s)}$ of signature $(p,q)$} defined as follows: the first $r$ rows of $\tilde\lambda_{(r,s)}$ contain the symbol $\smallyoung{+-}$; the next $s$ rows contain $\smallyoung{-+}$; the next $p-(r+s)$ rows contain $\smallyoung{+}$; the last $q-(r+s)$ rows contain $\smallyoung{-}$. The diagram $\tilde\lambda_{(r,s)}$ has $n-(r+s)$ left-justified rows, and it contains $p$ symbols ``$+$'' and $q$ symbols ``$-$''. \\
{\rm (b)} We represent an element $((r,s),\mu)\in\Pi_2$ by a {\em striped signed diagram} $\tilde\lambda_{((r,s),\mu)}$ defined as follows. Consider a grid pattern whose rows are numbered $1,2,\ldots$ from top to bottom and whose columns are numbered $-2,-1,0,1,2$ from left to right.
Then $\tilde\lambda_{((r,s),\mu)}$ is the diagram obtained from $\tilde\lambda_{(r,s)}$
by shifting horizontally the rows so that, for every $i\in\{1,\ldots,n-(r+s)\}$, the first box of the $i$-th row of $\tilde\lambda_{((r,s),\mu)}$ lies in the column $-\mu_i$.
\end{notation}



\begin{example}
{\rm (a)} Let $(p,q)=(4,5)$ and $(r,s)=(2,1)$. Then $\tilde\lambda_{(r,s)}=\smallyoung{+-,+-,-+,+,-,-}$. \\
{\rm (b)} Let $(p,q)=(4,5)$. The set $\Pi_2$ comprises three elements of the form $((2,1),\mu)$, which correspond to the following striped signed diagrams.
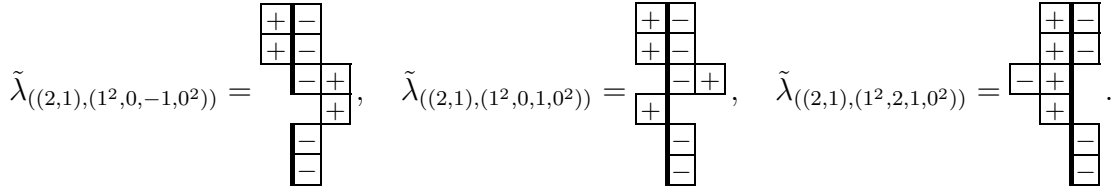
\[\tilde\lambda_{((2,1),(1^2,0,-1,0^2))}=\smallyoung{+\barm,+\barm,:\barm{+},::+,:\barm,:\barm}\,,\quad
\tilde\lambda_{((2,1),(1^2,0,1,0^2))}=\smallyoung{\pbar{-},\pbar{-},:\barm{+},\pbar,:\barm,:\barm}\,,\quad
\tilde\lambda_{((2,1),(1^2,2,1,0^2))}=\smallyoung{:+\barm,:+\barm,-\pbar,:\pbar,::\barm,::\barm}\,.\]
They are respectively of types (I), (II), (III). In the pictures the column number 0 is the first one on the right of the thick line.
\end{example}

\begin{remark}
Note that an element $((r,s),\mu)$ in $\Pi_2$ (or, equivalently, its representation $\tilde\lambda_{((r,s),\mu)}$) is actually characterized by the numbers $p,q,r,s$ and the type (I), (II), or (III): the sequence $\mu$ can be recovered from this information.
As it is formulated, Definition \ref{definition-2-1} is a faithful translation of \cite[Definition 4.1]{Johnson.2010.MR} to our situation (the objects defined in \cite[Definition 4.1]{Johnson.2010.MR} allow to classify a larger class of orbits). Note also that the sequence $\mu$ is needed in the dimension formula stated in Proposition \ref{proposition-2-5}\,{\rm (c)} below.
\end{remark}

The next proposition follows from \cite[Theorem 4.12 and Corollary 5.7]{Johnson.2010.MR}.

\begin{proposition}
\label{proposition-2-5}
{\rm (a)} For $(r,s)\in\Lambda_2$, let $\mathcal{O}_{(r,s)}^K$ be the set of elements $\left(\begin{smallmatrix}
0 & a \\ b & 0
\end{smallmatrix}\right)\in\mathcal{N}(\mathfrak{s})_2$ such that $\rk a=r$ and $\rk b=s$. Then $\mathcal{O}_{(r,s)}^K$ is a $K$-orbit of $\mathcal{N}(\mathfrak{s})_2$. Conversely every $K$-orbit of $\mathcal{N}(\mathfrak{s})_2$ is of the form $\mathcal{O}_{(r,s)}^K$ for a unique pair $(r,s)\in\Lambda_2$. \\
{\rm (b)} Let $\mathfrak{O}_{((r,s),\mu)}^K$ be the set of pairs $(L,\left(\begin{smallmatrix}
0 & a \\ b & 0
\end{smallmatrix}\right))\in\mathbb{P}(V_1)\times\mathcal{O}_{(r,s)}^K$ such that: 
\[
\left\{
\begin{array}{ll}
L\subset\mathrm{Im}\,a & \mbox{if $((r,s),\mu)$ is of type (I)};  \\[1mm]
L\subset\ker b,\ L\not\subset\mathrm{Im}\,a & \mbox{if $((r,s),\mu)$ is of type (II)}; \\[1mm]
L\not\subset\ker b & \mbox{if $((r,s),\mu)$ is of type (III)}.
\end{array}
\right.
\]
Then $\mathfrak{O}_{((r,s),\mu)}^K$ is a $K$-orbit of $\mathbb{P}(V_1)\times\mathcal{N}(\mathfrak{s})_2$. Conversely every $K$-orbit of $\mathbb{P}(V_1)\times\mathcal{N}(\mathfrak{s})_2$
is of the form $\mathfrak{O}_{((r,s),\mu)}^K$ for a unique pair $((r,s),\mu)\in\Pi_2$. \\
{\rm (c)} $\displaystyle \dim \mathcal{O}_{(r,s)}^K=(r+s)(n-(r+s))$ and $\displaystyle \dim \mathfrak{O}_{((r,s),\mu)}^K=(r+s)(n-(r+s))+\sum_{i=1}^{n-(r+s)}\left\lceil\frac{\mu_i}{2}\right\rceil-1$.
\end{proposition}

\begin{remark}
The description of $K$-orbits presented here is adapted from the one in \cite[\S4--5]{Johnson.2010.MR}, from which it is however slightly different. 
Indeed, the $K$-orbits of  $V_1\times\mathcal{N}(\mathfrak{s})_2$ are classified in \cite{Johnson.2010.MR} whereas we are concerned with the product $\mathbb{P}(V_1)\times\mathcal{N}(\mathfrak{s})_2$.
\end{remark}

\begin{corollary}\label{corollary-2-6}
Let $\Pi_2^k=\{((r,s),\mu)\in\Pi_2:r+s\leq\min\{k,n-k\}\}$. Then we have 
\[\mathfrak{E}=\bigsqcup_{((r,s),\mu)\in\Pi_2^k}\mathfrak{O}_{((r,s),\mu)}^K\]
which is the decomposition of the exotic nilpotent cone $\mathfrak{E}$ of Notation \ref{notation-1-5}\,{\rm (c)} into $K$-orbits.
\end{corollary}

\subsection{Main results}\label{section-2-2}

The setting is explained in Notation \ref{notation-1-5}.
In particular we consider the conormal variety
$\mathcal{Y}\subset\Gr_k(V)\times\mathbb{P}(V_1)\times\mathcal{N}(\mathfrak{g})$ and the exotic nilpotent cone $\mathfrak{E}=\mathbb{P}(V_1)\times\mathcal{N}(\mathfrak{s})_2^k\subset\mathbb{P}(V_1)\times\mathcal{N}(\mathfrak{s})$.
Recall the projection $\mathfrak{g}\to\mathfrak{s}$, $x\mapsto x^{-\theta}$ (see Notation \ref{notation-1-5}).
As mentioned in Section \ref{section-1-5}, the following statement is valid. We prove it in~Section~\ref{section-3}.

\begin{proposition}
\label{proposition-2-8}
The map
\[\pi:\mathcal{Y}\to\mathfrak{E},\ (W,L,x)\mapsto (L,x^{-\theta})\]
is well defined and surjective.
\end{proposition}

Recall that an element $((r,s),\mu)\in\Pi_2^k$ can be of type (I), (II), or (III), according to Definition \ref{definition-2-1}. We introduce subtypes (II)$^0$ and (II)$^*$.

\begin{definition}
\label{definition-2-9}
We say that an element $((r,s),\mu)\in\Pi_2^k$ is of type (II)$^*$ if it is of type (II) and satisfies
\begin{equation}
\label{5}
q=r+s\leq p-2\quad\mbox{and}\quad q<\min\{k,n-k\}.
\end{equation}
We say that $((r,s),\mu)$ is of type (II)$^0$ if it is of type (II) without satisfying (\ref{5}).
\end{definition}

\begin{remark}
{\rm (a)}
The condition $q=r+s\leq p-2$ means that the striped signed diagram $\tilde\lambda_{((r,s),\mu)}$ contains no row of the form $\smallyoung{-}$ and contains at least two rows of the form $\smallyoung{+}$.
Furthermore, $((r,s),\mu)$ being of type (II), we have that all rows of $\tilde\lambda_{((r,s),\mu)}$ of the form $\smallyoung{-+}$ lie on the right of the thick line whereas all rows  $\smallyoung{+}$ lie on the left. \\
{\rm (b)} The condition $r+s<\min\{k,n-k\}$ means that the nilpotent $G$-orbit $G\mathfrak{O}^K_{((r,s),\mu)}$ has positive codimension inside the set $\mathcal{N}(\mathfrak{g})_2^k:=\{z\in\mathcal{L}(V):z^2=0,\ \rk z\leq\min\{k,n-k\}\}$, which is the closure of the Richardson nilpotent $G$-orbit corresponding to a maximal parabolic subgroup of type ($k,n-k)$ (see Section \ref{section-3-2}). \\
{\rm (c)} In the case where $p\leq\max\{k,n-k,q+1\}$, or equivalently ($q>p-2$ or $q\geq\min\{k,n-k\}$), relation (\ref{5}) cannot occur, hence
every element $((r,s),\mu)$ of type (II) is of type (II)$^0$. If $p>\max\{k,n-k,q+1\}$, then we can always find elements of type (II)$^*$.
\end{remark}

\begin{example}
\label{example-2-10}
Assume that $(p,q)=(3,1)$. \\
{\rm (a)}
The set $\Pi_2^2$
comprises five elements, which correspond to the following striped signed diagrams (for each diagram we indicate its type).
\begin{equation}
\label{6}
\smallyoung{+\barm,::+,::+}\;\;{}^{\mbox{\scriptsize (I)}}\,,\quad
\smallyoung{+\barm,\pbar,\pbar}\;\;{}^{\mbox{\scriptsize (II)}^*}\,,\quad
\smallyoung{:\barm+,\pbar,\pbar}\;\;{}^{\mbox{\scriptsize (II)}^*}\,,\quad
\smallyoung{-\pbar,:\pbar,:\pbar}\;\;{}^{\mbox{\scriptsize (III)}}\,,\quad
\smallyoung{\pbar,\pbar,\pbar,:\barm}\;\;{}^{\mbox{\scriptsize (II)}^0}\,.
\end{equation}
{\rm (b)}
The set $\Pi_2^1$ coincides with $\Pi_2^2$. In particular it corresponds to the same list of striped signed diagrams as in (\ref{6}), with the single difference that the second and the third diagrams in the list are of type $\mbox{(II)}^0$ when viewed as elements of $\Pi_2^1$. There is no element of type $\mbox{(II)}^*$ in $\Pi_2^1$.

\end{example}

Let us recall the situation once again.  
Since there are only finitely many exotic nilpotent orbits and since the map $ \pi : \mathcal{Y} \to \mathfrak{E} $ is $K$-equivariant, 
an irreducible component 
$ \mathcal{C} $ of the conormal variety $ \mathcal{Y} $
is mapped densely by $\pi$ to the closure of an exotic nilpotent orbit $\mathfrak{O}\subset\mathfrak{E}$.
The explicit correspondence between $ \mathcal{C} $ and $ \mathfrak{O} $ is given by 
the following theorem.  
Then, Theorem~\ref{theorem-1} comes as its consequence 
(taking also Corollary~\ref{corollary-2-6} into account).

\begin{theorem}\label{theorem-2}
We assume the notation above; in particular 
$ \Pi_2^k $ denotes the set defined in Corollary~\ref{corollary-2-6}.
Let $\mathfrak{O}=\mathfrak{O}^K_{((r,s),\mu)}$ be the $K$-orbit of $\mathfrak{E}$ corresponding to the element $((r,s),\mu)\in\Pi_2^k$. 

\medskip

\noindent
{\rm (a)} If $((r,s),\mu)$ is of type $(I), (II)^0$, or $(III)$, 
then there is a unique component $\mathcal{C}_{\mathfrak{O}}$ of $\mathcal{Y}$ such that
$\mathcal{C}_{\mathfrak{O}}\subset\overline{\pi^{-1}(\mathfrak{O})}$. 
\\
{\rm (b)} If $((r,s),\mu)$ is of type $(II)^*$, 
then there are exactly two components $\mathcal{C}_\mathfrak{O}^1$ and 
$\mathcal{C}_\mathfrak{O}^{2}$ of $\mathcal{Y}$ such that 
$\mathcal{C}_{\mathfrak{O}}^i\subset\overline{\pi^{-1}(\mathfrak{O})}$ for $i\in\{1,2\}$. 
\\
{\rm (c)} Every irreducible component of the conormal variety $\mathcal{Y}$ is of the form 
$\mathcal{C}_{\mathfrak{O}}$, $\mathcal{C}_{\mathfrak{O}}^1$, or $\mathcal{C}_{\mathfrak{O}}^2$ 
for a unique $K$-orbit $\mathfrak{O}$ of the exotic nilpotent cone $\mathfrak{E}$.
\end{theorem}

Since the irreducible components of the conormal variety $ \mathcal{Y} $ correspond 
bijectively to the $ K $-orbits of the double flag variety $ \mathfrak{X} $, 
we get a map from $ \mathfrak{X}/K $ to $ \mathfrak{E}/K $ (see Corollary~\ref{corollary-1}).

\begin{corollary}
\label{corollary-2-11}
The map $\Phi:\mathfrak{X}/K\to\mathfrak{E}/K$ satisfies the following conditions: 
\\
{\rm (a)} If $((r,s),\mu)\in\Pi_2^k$ is of type $(I), (II)^0, (III)$, 
then $\Phi^{-1}(\mathfrak{O}^K_{((r,s),\mu)})$ is a singleton. \\
{\rm (b)} If $((r,s),\mu)$ is of type $(II)^*$, then $\Phi^{-1}(\mathfrak{O}^K_{((r,s),\mu)})$ is a pair. \\
{\rm (c)} $\Phi$ is a bijection if and only if $p\leq\max\{k,n-k,q+1\}$.
\end{corollary}

Furthermore we can make the correspondence $\Phi$ of Corollary \ref{corollary-2-11} explicit. Some examples are discussed in the following subsection.

\subsection{Parametrization of $K$-orbits in $\Xfv$}\label{section-2-3}

We start this section with a description of the orbits of $K=\GL(V_1)\times \GL(V_2)$ in the exotic Grassmannian $\Xfv=\Gr_k(V)\times\mathbb{P}(V_1)$.

\begin{definition}
Let $\delta$ be the diagram formed by two columns consisting of respectively  $p$ and $q$ boxes. \\
{\rm (a)}
We call {\it $(1,2)$-tableau of shape $(p,q)$ and weight $k$} 
a filling of the boxes of $\delta$
by the numbers $0,1,2$ in such a way that the columns are nondecreasing from top to bottom, the number of $1$'s is the same in both columns, and the half sum of the entries of the tableau is equal to $k$. \\
{\rm (b)} A {\it marked} $(1,2)$-tableau is a pair of the form $(\tau,i)$ where $\tau$ is a $(1,2)$-tableau and $i$ is an entry occurring in the first column of $\tau$.
We denote by $\Theta_2^k$ the set of marked $(1,2)$-tableaux of shape $(p,q)$ and weight $k$.
\end{definition}

\begin{example} 
\label{example-2-14}
Assume that $(p,q)=(3,1)$. \\
{\rm (a)} The set $\Theta_2^2$ consists of seven elements enumerated as follows: \[\!\left(\smallyoung{00,2,2},0\right)\!,\ \!\left(\smallyoung{00,2,2},2\right)\!,\ \!\left(\smallyoung{02,0,2},0\right)\!,\ \!\left(\smallyoung{02,0,2},2\right)\!,\ \!\left(\smallyoung{01,1,2},0\right)\!,\ \!\left(\smallyoung{01,1,2},1\right)\!,\ \!\left(\smallyoung{01,1,2},2\right)\!.\]
{\rm (b)} The set $\Theta_2^1$ consists of five elements:
\[\!\left(\smallyoung{00,0,2},0\right)\!,\ \!\left(\smallyoung{00,0,2},2\right)\!,\ \!\left(\smallyoung{02,0,0},0\right)\!,\ \!\left(\smallyoung{01,0,1},0\right)\!,\ \!\left(\smallyoung{01,0,1},1\right)\!.\]
\end{example}

The following result is proved in Section \ref{section-6-1}.

\begin{proposition}
\label{proposition-2-16}
For $(\tau,i)\in \Theta_2^k$, let $\mathbb{O}_{(\tau,i)}\subset\Xfv$ be the subset formed by pairs $(W,L)\in\Gr_k(V)\times\mathbb{P}(V_1)$ such that:
\begin{equation}
\label{7}
\left\{
\begin{array}{ll}
\multicolumn{2}{l}{\dim W\cap V_j=\mbox{{\rm number of $2$'s in the $j$-th column of $\tau$,} for $j\in\{1,2\}$};} \\[1mm]
L\subset W & \mbox{if $i=2$};  \\[1mm]
L\subset W+V_2,\ L\not\subset W & \mbox{if $i=1$}; \\[1mm]
L\not\subset W+V_2 & \mbox{if $i=0$}.
\end{array}
\right.
\end{equation}
Then $\mathbb{O}_{(\tau,i)}$ is a $K$-orbit of $\Xfv$. Conversely every $K$-orbit of $\Xfv$ is of the form $\mathbb{O}_{(\tau,i)}$ for a unique marked $(1,2)$-tableau $(\tau,i)\in\Theta_2^k$.
\end{proposition}

\begin{remark}
As a byproduct of Proposition \ref{proposition-2-16}, we can see that $K$-orbits ($K=\GL(V_1)\times\GL(V_2)$) in the Grassmannian $\Gr_k(V)$ are parametrized by $(1,2)$-tableaux $\tau$ of shape $(p,q)$ and weight $k$ (without marking): an element $W\in\Gr_k(V)$ belongs to the $K$-orbit attached to $\tau$ if it fulfills the first line of (\ref{7}). This is actually a special case of the classification of $K$-orbits in the full flag variety $\GL(V)/B$ (for $B$ a Borel subgroup); see \cite{Matsuki.Oshima.1990,Yamamoto.1997}, and \cite{Wyser.2015} for the degeneracy order of these orbits.
By Proposition~\ref{proposition-1} the enhanced full flag variety $\mathbb{P}(V_1)\times (\GL(V)/B)$ has also finitely many $K$-orbits, but as far as we know there is no combinatorial parametrization of these orbits; an abstract parametrization is given in \cite{HNOO.2013}.
\end{remark}

We consider the surjective map $\Phi:\Xfv/K\to\mathfrak{E}/K$ involved in Corollaries \ref{corollary-1} and \ref{corollary-2-11}.
Recall that the image by $\Phi$ of a $K$-orbit $\mathbb{O}=\mathbb{O}_{(\tau,i)}\subset\Xfv$ (with $(\tau,i)\in\Theta_2^k$) is the unique $K$-orbit $\Phi(\mathbb{O})=\mathfrak{O}^K_{((r,s),\mu)}\subset\mathfrak{E}$
(with $((r,s),\mu)\in\Pi_2^k$) such that  
 $\pi^{-1}(\mathfrak{O}^K_{((r,s),\mu)})\cap T^*_\mathbb{O}\Xfv$ is open in $T^*_\mathbb{O}\Xfv$  (from Section \ref{section-1-5}).
The surjection $\Phi$ induces a surjective map $\phi:\Theta_2^k\to\Pi_2^k$ between the parameter sets. 
In Examples \ref{example-2-17}--\ref{example-2-18} we calculate the map $\phi$ in three situations.
The proofs of these examples can be found in Section \ref{section-6-2}.

\begin{example}
\label{example-2-17}
Let $(p,q)=(3,1)$ and $k\in\{1,2\}$.
In that case the sets $\Pi_2^k$ and $\Theta_2^k$ are described in Examples \ref{example-2-10} and \ref{example-2-14}.
The next table determines the map $\phi:\Theta_2^k\to\Pi_2^k$.
\[
\renewcommand{\arraystretch}{1.5}
\begin{array}{c||c|c|c|c|c|}
\mbox{\tiny $((r,s),\mu)\in \Pi_2^k$}
\vphantom{\verysmallyoung{\spbar,\spbar,\spbar,:\sbarm,\spbar}}
& \verysmallyoung{+\sbarm,::+,::+} &  \verysmallyoung{+\sbarm,\spbar,\spbar} &
\verysmallyoung{:\sbarm+,\spbar,\spbar} &
\verysmallyoung{-\spbar,:\spbar,:\spbar} &
\verysmallyoung{\spbar,\spbar,\spbar,:\sbarm}
 \\
\hline
\begin{array}{c}
\mbox{\tiny $\phi^{-1}(((r,s),\mu))$} \\ \mbox{\tiny for $k=1$}
\end{array} & \mbox{\tiny $\!\left(\verysmallyoung{00,0,2},2\right)\!$} & \mbox{\tiny $\!\left(\verysmallyoung{00,0,2},0\right)\!$} & \mbox{\tiny $\!\left(\verysmallyoung{01,0,1},1\right)\!$}
& \mbox{\tiny $\!\left(\verysmallyoung{02,0,0},0\right)\!$}
& \mbox{\tiny $\!\left(\verysmallyoung{01,0,1},0\right)\!$} \\
\hline
\begin{array}{c}
\mbox{\tiny $\phi^{-1}(((r,s),\mu))$} \\ \mbox{\tiny for $k=2$}
\end{array} & \mbox{\tiny $\!\left(\verysmallyoung{01,1,2},2\right)\!$} & \mbox{\tiny $\!\left(\verysmallyoung{00,2,2},2\right)\!$}, \mbox{\tiny $\!\left(\verysmallyoung{00,2,2},0\right)\!$} & \mbox{\tiny $\!\left(\verysmallyoung{02,0,2},2\right)\!$}, \mbox{\tiny $\!\left(\verysmallyoung{01,1,2},1\right)\!$} & \mbox{\tiny $\!\left(\verysmallyoung{02,0,2},0\right)\!$} & \mbox{\tiny $\!\left(\verysmallyoung{01,1,2},0\right)\!$} \\
\hline
\end{array}
\renewcommand{\arraystretch}{1}
\]
\end{example}

\begin{example}
\label{example-2-18}
Let $p=q=k=\frac{n}{2}=:m$.
The map $\phi:\Theta_2^k\to\Pi_2^k$ is bijective in this case, and the image by $\phi$ of the element $(\tau,i)\in\Theta_2^k$ is the pair $(\tau,i)=((r,s),\mu)$ given by:
\[
\left\{
\begin{array}{l}
r=\mbox{number of $2$'s in the first column of $\tau$}; \\[1mm]
s=\mbox{number of $2$'s in the second column of $\tau$}; \\[1mm]
\mu=\left\{\begin{array}{lll}
(1^r,0^s,(-1)^{m-(r+s)},0^{m-(r+s)}) & \mbox{if $i=2$} & \mbox{($((r,s),\mu)$ is of type (I) in this case)}; \\[1mm]
(1^r,0^s,1^{m-(r+s)},0^{m-(r+s)}) & \mbox{if $i=1$} & \mbox{($((r,s),\mu)$ is of type (II) in this case)}; \\[1mm]
(1^r,2^s,1^{m-(r+s)},0^{m-(r+s)}) & \mbox{if $i=0$} & \mbox{($((r,s),\mu)$ is of type (III) in this case)}. \\[1mm]
\end{array}\right.
\end{array}
\right.
\]
For instance, the next table summarizes the map $\phi$ for $p=q=k=2$.
\[
\renewcommand{\arraystretch}{2}
\begin{array}{c||c|c|c|c|c|c|c|c|c|}
\mbox{\tiny $(\tau,i)$}
 & \mbox{\tiny $\!\left(\verysmallyoung{20,20},2\right)\!$}
& \mbox{\tiny $\!\left(\verysmallyoung{00,22},2\right)\!$} & \mbox{\tiny $\!\left(\verysmallyoung{00,22},0\right)\!$}\! & \!\mbox{\tiny $\!\left(\verysmallyoung{02,02},0\right)\!$}
& \mbox{\tiny $\!\left(\verysmallyoung{10,21},2\right)\!$}
& \mbox{\tiny $\!\left(\verysmallyoung{10,21},1\right)\!$}
& \mbox{\tiny $\!\left(\verysmallyoung{01,12},1\right)\!$}
& \mbox{\tiny $\!\left(\verysmallyoung{01,12},0\right)\!$}
& \mbox{\tiny $\!\left(\verysmallyoung{11,11},1\right)\!$} \\[.5ex]
\hline
\mbox{\tiny $\phi(\tau,i)$} 
\vphantom{\verysmallyoung{\spbar,\spbar,:\sbarm,:\sbarm,\spbar}} & 
\verysmallyoung{+\sbarm,+\sbarm} & \verysmallyoung{+\sbarm,:\sbarm:+} &  \verysmallyoung{:+\sbarm,-\spbar} &
\verysmallyoung{-\spbar,-\spbar} &
\verysmallyoung{+\sbarm,::+,:\sbarm} &
\verysmallyoung{+\sbarm,\spbar,:\sbarm} &
\verysmallyoung{:\sbarm+,\spbar,:\sbarm} &
\verysmallyoung{-\spbar,:\spbar,::\sbarm} &
\verysmallyoung{\spbar,\spbar,:\sbarm,:\sbarm}
\end{array}
\renewcommand{\arraystretch}{1}
\]
In the Appendix we describe the map $\phi$ in the case $p=q=k=3$.
\end{example}

\section{Review on Spaltenstein varieties}

In this section we introduce notation and review some basic facts on nilpotent orbits, partial flag varieties, and Spaltenstein varieties for the group $G=\GL(V)$.

\subsection{Nilpotent orbits}
As in Notation \ref{notation-1-5} we
write $G=\GL(V)$ and $\mathfrak{g}=\mathcal{L}(V)$.
The nilpotent cone $\mathcal{N}:=\mathcal{N}(\mathfrak{g})$ is the set of nilpotent endomorphisms $x\in\mathcal{L}(V)$. 
There are only finitely many {\it nilpotent $G$-orbits} thanks to the theory of Jordan normal forms.

A $G$-orbit through $x\in\mathcal{N}$ is characterized by the sizes of the Jordan blocks $\lambda(x):=(\lambda_1\geq\ldots\geq\lambda_r)$ of $x$, which form a partition of $n=\dim V$, or equivalently a Young diagram of size $n$ (with rows of lengths $\lambda_1,\ldots,\lambda_r$). We denote by
\[\mathcal{O}^G_\lambda:=\{x\in\mathcal{N}:\lambda(x)=\lambda\}\subset\mathfrak{g}\]
the nilpotent $G$-orbit corresponding to the partition $\lambda\vdash n$.

Given another partition $\mu=(\mu_1\geq\ldots\geq\mu_s)\vdash n$, recall that
\[\overline{\mathcal{O}^G_\mu}\subset\overline{\mathcal{O}^G_\lambda}\quad\Longleftrightarrow\quad\mu\preceq\lambda\]
where $\mu\preceq\lambda$ stands for the dominance order, which means that $\mu_1+\ldots+\mu_i\leq\lambda_1+\ldots+\lambda_i$ for all $i\in\{1,2,\ldots,\min\{r,s\}\}$.  
For the proof, and more detailed properties of nilpotent orbits used below, 
we refer the readers to \cite{Collingwood.McGovern.1993}. 

The {\it dual partition} $\lambda^*=(\lambda^*_1,\ldots,\lambda^*_{\lambda_1})$ is the partition obtained from $\lambda$ by Young diagram transposition, i.e., such that $\lambda_1^*,\ldots,\lambda_{\lambda_1}^*$ are the lengths of the columns of $\lambda$. Note that the duality $\mathcal{O}^G_\lambda\mapsto\mathcal{O}^G_{\lambda^*}$ reverses the inclusion relations between orbit closures.

Finally we recall the dimension formula
\[\dim\mathcal{O}_{\lambda}^G=n^2-\sum_{i=1}^{\lambda_1}(\lambda_i^*)^2.\]

\begin{example}
If $k\in\{0,\ldots,\lfloor\frac{n}{2}\rfloor\}$, then
$\mathcal{O}^G_{(n-k,k)^*}=\{x\in\mathcal{L}(V):x^2=0\ \mbox{and}\ \rk x=k\}$
and
$\overline{\mathcal{O}^G_{(n-k,k)^*}}=\{x\in\mathcal{L}(V):x^2=0\ \mbox{and}\ \rk x\leq k\}=\bigcup_{\ell=0}^k\mathcal{O}^G_{(n-\ell,\ell)^*}$.
Furthermore $\dim\mathcal{O}^G_{(n-k,k)^*}=2k(n-k)$.
\end{example}

\subsection{Parabolic subgroups}
\label{section-3-2}
If $\underline{d}=(d_1,\ldots,d_r)$ is a composition of $n$ (i.e., unordered partition of $ n $), 
we denote by $P_{\underline{d}}\subset G$ the corresponding standard parabolic subgroup, i.e., the subgroup of blockwise upper triangular matrices with diagonal blocks of sizes $d_1,\ldots,d_r$. Its Lie algebra $\mathfrak{p}_{\underline{d}}\subset \mathfrak{g}\cong M_n(\mathbb{C})$ is given by
\[\mathfrak{p}_{\underline{d}}=\left\{x=\left(\begin{array}{ccc}
x_{1,1} & \ldots & x_{1,r} \\
 & \ddots & \vdots \\
 0 & & x_{r,r}
\end{array}\right):x_{i,j}\in M_{d_i,d_j}(\mathbb{C}) \;\; (1 \leq i \leq j \leq r)  \right\}\]
and the corresponding nilradical is $\mathfrak{n}_{\mathfrak{p}_{\underline{d}}}=\{x\in\mathfrak{p}_{\underline{d}}:x_{i,i}=0\ \forall i\}$.  

Every parabolic subgroup $P\subset G=\GL(V)$ is conjugated to $P_{\underline{d}}$ for some $\underline{d}$. The partial flag variety $G/P$ can then be regarded as the variety of partial flags
\[G/P=\mathcal{F}_{\underline{d}}:=\{(W_0=0\subset W_1\subset \ldots\subset W_r= V): \dim W_i/W_{i-1}=d_i\}.\]
The {\it Richardson nilpotent orbit} corresponding to $P$, denoted by $\mathcal{O}_P^G$,
is by definition the open $G$-orbit of the set 
\[G\mathfrak{n}_{\mathfrak{p}_{\underline{d}}}=G\cdot \{x\in \mathcal{L}(V):x(W_i)\subset W_{i-1}\ \forall i\geq 1\}\subset\mathcal{N}\]
(with some/any $(W_0,\ldots,W_r)\in G/P$).
Letting $\lambda(\underline{d})$ be the partition of $n$ obtained by arranging the sequence $\underline{d}$ in nonincreasing order, we have 
\[\mathcal{O}_P^G=
\mathcal{O}^G_{\lambda(\underline{d})^*}.\]

The parabolic subgroup $P$ is {\it maximal} if it is conjugated to $P_{(k,n-k)}$ for some $k\in\{1,\ldots,n-1\}$. Then, the partial flag variety $G/P$ coincides with the Grassmann variety
$\Gr_k(V)$, while
its Richardson orbit is $\mathcal{O}^G_{(\max\{k,n-k\},\min\{k,n-k\})^*}$.

Following \cite{Finkelberg.Ginzburg.Travkin.2009} we call  $P$ {\it mirabolic} if it is conjugated to $P_{(1,n-1)}$ or $P_{(n-1,1)}$. Equivalently $G/P$ is isomorphic to the projective space $\mathbb{P}(V)$.

Every parabolic subgroup $Q\subset K=\GL(V_1)\times\GL(V_2)$ is of the form $Q=Q_1\times Q_2$, where $Q_1,Q_2$ are parabolic subgroups of $\GL(V_1)$, $\GL(V_2)$.
Moreover $Q$ is maximal if and only if ($Q_1$ is maximal and $Q_2=\GL(V_2)$) or ($Q_1=\GL(V_1)$ and $Q_2$ is maximal). We call $Q$  {\it mirabolic} if it is maximal and $Q_1$ or $Q_2$ is mirabolic, so that $K/Q\cong\mathbb{P}(V_1)$ or $\mathbb{P}(V_2)$.

\subsection{Spaltenstein varieties}
\label{section-3-3}
Let $P=P_{\underline{d}}\subset GL(V)$ be a standard parabolic subgroup, corresponding to the composition $\underline{d}=(d_1,\ldots,d_r)$.
For $x\in \mathcal{N}$, the variety
\[\mathcal{F}_{x,\underline{d}}:=\{(W_0,\ldots,W_r)\in \mathcal{F}_{\underline{d}}(=G/P):x(W_i)\subset W_{i-1}\ \forall i\geq 1\}\]
is called a {\it Spaltenstein variety}.
It is nonempty if and only if $x$ belongs to the Richardson orbit closure $\overline{\mathcal{O}_P^G}$.

The variety $\mathcal{F}_{x,\underline{d}}$ is 
determined by the Jordan form $\lambda=\lambda(x)$ of $x$ up to isomorphism and we denote
$\mathcal{F}_{\lambda,\underline{d}}:=\mathcal{F}_{x,\underline{d}}$ by abuse of notation.
We have
\[
\dim\mathcal{F}_{\lambda,\underline{d}}=\dim G/P-\frac{1}{2}\dim\mathcal{O}^G_\lambda.
\]
The variety $\mathcal{F}_{\lambda,\underline{d}}$ is equidimensional and there is an explicit bijection
\begin{equation}
\label{8}
\mathrm{SSTab}(\lambda,\underline{d})\stackrel{\sim}{\to}\mathrm{Irr}(\mathcal{F}_{\lambda,\underline{d}})
\end{equation}
between the set of irreducible components of $\mathcal{F}_{\lambda,\underline{d}}$
and the set
$\mathrm{SSTab}(\lambda,\underline{d})$ of
{\it semistandard tableaux of shape $\lambda$ and weight $\underline{d}$}, i.e., numberings of the Young diagram $\lambda$
with $d_i$ entries equal to $i$ (for all $i$) and such that the entries are increasing to the right along the rows and nondecreasing to the bottom along the columns. We refer to \cite{Spaltenstein.1982} for more information on Spaltenstein varieties.

In the case of a maximal parabolic subgroup  $P=P_{(k,n-k)}$ (for $1\leq k\leq n-1$),
the Spaltenstein variety $\mathcal{F}_{\lambda,(k,n-k)}$ turns out to be irreducible.
This fact is a consequence of (\ref{8}) since the set 
$\mathrm{SSTab}(\lambda,(k,n-k))$ is a singleton (if nonempty).
Actually, it can simply be recovered as follows.

\begin{proposition}\label{proposition-2}
Let $k\in\{1,\ldots,n-1\}$
and let $x\in\mathcal{O}^G_{(n-\ell,\ell)^*}$ with $0\leq \ell\leq\min\{k,n-k\}$,
so that the Spaltenstein variety $\mathcal{F}_{x,(k,n-k)}$ is nonempty.
Then we have
\[\mathcal{F}_{x,(k,n-k)}=\{W\in\Gr_k(V):\mathrm{Im}\,x\subset W\subset \ker x\}.\]
Thus $\mathcal{F}_{x,(k,n-k)}$ is isomorphic to the Grassmann variety $\Gr_{k-\ell}(\ker x/\mathrm{Im}\,x)$.
In particular, it is smooth and irreducible of dimension $(k-\ell)(n-k-\ell)$.
\end{proposition}

\begin{proof}
The proof easily follows from the definition of $\mathcal{F}_{x,(k,n-k)}$ and the general properties of Grassmann varieties.
\end{proof}

\section{Definition of the map $\pi$}

\label{section-3}

In this section, as before,
we consider the symmetric pair $(G,K)$ with $G=\GL(V)$ and $K=\GL(V_1)\times \GL(V_2)$
(see Notation \ref{notation-1-5}).
The goal of this section is to show Proposition~\ref{proposition-2-8}.
We actually show a more general property (see Proposition \ref{proposition-4-2} and Corollary \ref{corollary-4-3}).

\subsection{Projection of nilpotent elements}
As in Notation \ref{notation-1-5new} and \ref{notation-1-5}\,{\rm (b)}, we have the decomposition $\mathfrak{g}:=\Lie(G)=\mathfrak{k}\oplus\mathfrak{s}$,
and in particular every $x\in\mathfrak{g}$ is uniquely written as 
$x=x^\theta+x^{-\theta}$ with $x^\theta\in\mathfrak{k}$ and $x^{-\theta}\in\mathfrak{s}$.
As in Notation \ref{notation-1-5new}, we denote by $\mathcal{N}=\mathcal{N}(\mathfrak{g})$, $\mathcal{N}(\mathfrak{k})$, $\mathcal{N}(\mathfrak{s})$ the subsets of nilpotent elements in $\mathfrak{g}$, $\mathfrak{k}$, and $\mathfrak{s}$.

We first observe that the property that 
$x$ and $x^\theta$ are nilpotent does not imply that $x^{-\theta}$ is nilpotent in general.

\begin{example}
Assume $p=q=3$, i.e., $G=\GL_6(\mathbb{C})$ and $K=\GL_3(\mathbb{C})\times\GL_3(\mathbb{C})$. Take 
\begin{equation}
x = \begin{pmatrix}
a & 1_3 
\\
c & -a 
\end{pmatrix}
\quad 
\text{ where } 
a = 
\begin{pmatrix}
0 & 1 & 0 \\
0  & 0 & -1 \\
0  & 0  & 0
\end{pmatrix}, \quad
c = 
\begin{pmatrix}
0 & 0 & 1 \\
0 & 0 & 0 \\
1 & 0 & 0
\end{pmatrix}.
\end{equation}
Then, $x$ is nilpotent (since $x^4=0$), $x^\theta$ is nilpotent (clearly), but $x^{-\theta}$ is not nilpotent (its minimal polynomial is $P(t)=t^2(t^4-1)$).

\end{example}

We consider parabolic subgroups $P\subset G$ and $Q\subset K$, and the corresponding Richardson nilpotent orbits $\mathcal{O}_P^G\subset\mathfrak{g}$ and $\mathcal{O}_Q^K\subset\mathfrak{k}$
(see Section \ref{section-3-2}); 
thus $\overline{\mathcal{O}_P^G}=G\mathfrak{n}_\mathfrak{p}$ and $\overline{\mathcal{O}_Q^K}=K\mathfrak{n}_\mathfrak{q}$, 
where $\mathfrak{n}_\mathfrak{p}$ and $\mathfrak{n}_\mathfrak{q}$ 
denote the nilradicals of the parabolic subalgebras 
$\mathfrak{p}:=\Lie(P)\subset\mathfrak{g}$ and $\mathfrak{q}:=\Lie(Q)\subset\mathfrak{k}$, respectively.

In this section, unless otherwise notified, we make no assumption on $P$ and $Q$, so that unlike in Notation \ref{notation-1-5} the double flag variety $\Xfv:=G/P\times K/Q$ is not necessarily of the form of an exotic Grassmannian.
The conormal variety $\mathcal{Y}\subset G/P\times K/Q\times\mathfrak{g}$
is described in formula (\ref{eq:conormal-variety-Y}).
Two sufficient conditions for $\Xfv$ to be of finite type (hence for $\mathcal{Y}$ to be equidimensional) are given in Proposition~\ref{proposition-1}, 
and these are precisely the two situations considered in the following statement.

\begin{proposition}\label{proposition-4-2}
Assume that at least one of the following two conditions is satisfied:
\begin{itemize}
\item[(a)] $Q\subset K$ is mirabolic, or
\item[(b)] $P\subset G$ is maximal. 
\end{itemize}
Then, for every $x\in\mathfrak{g}$, the condition that  
$x\in\overline{\mathcal{O}_P^G}$ and $x^\theta\in\overline{\mathcal{O}_Q^K}$
implies $x^{-\theta}$ is nilpotent.
\end{proposition}

\begin{proof}
First assume that we are in Case (a), i.e., $ Q $ is mirabolic.  
Without loss of generality we may take $Q=Q_1 \times GL(V_2)$, where 
$Q_1$ is the stabilizer of a line in $V_1$.  
In this situation the inclusion $x^\theta\in\overline{\mathcal{O}_Q^K}$ simply means that
\begin{equation}
\label{eq:x-preserves-line}
\mathrm{Im}\,x^\theta\subset L\subset\ker x^\theta
\end{equation}
for some line $L\subset V_1$.
Let us write
\begin{equation}\label{eq:x-is-ABCD}
x=\left(\begin{array}{cc}
a & b \\ c & d
\end{array}\right)
\end{equation}
with $a\in M_p(\mathbb{C})$, $b\in M_{p,q}(\mathbb{C})$, $c\in M_{q,p}(\mathbb{C})$, and
$d\in M_q(\mathbb{C})$. Relation \eqref{eq:x-preserves-line}
implies that
$d=0$ and $a$ is a nilpotent matrix of rank $\leq 1$.
Moreover, by assumption, the matrix $x$ is also nilpotent.
For $\alpha\in\mathbb{C}$, we set 
\[
x(\alpha)=\left(\begin{array}{cc}
\alpha a & b \\ c & 0
\end{array}\right) . 
\]
So, $x=x(1)$, and
we need to show that the matrix $x^{-\theta}=x(0)$ is also nilpotent.
To do this, we compute the characteristic polynomial of $x(\alpha)$,
\begin{eqnarray*}
\det(t1_{p+q}-x(\alpha)) & = & \det\left(\begin{array}{cc}t1_p-\alpha a & -b \\ -c & t1_q\end{array}\right) \\
 & = & \det\left[\left(\begin{array}{cc}1_p & - t^{-1}b \\ 0 & 1_q\end{array}\right) 
 \left(\begin{array}{cc}t1_p-\alpha a-t^{-1}bc & 0 \\ -c & t1_q\end{array}\right)\right] \\
 & = & t^{q-p}\det\big(t(t1_p-\alpha a)-bc\big).
\end{eqnarray*}
Since $\rk a\leq 1$, up to conjugating by an element of $\GL_p(\mathbb{C})$, we may assume
that the matrix $a$ has at most one nonzero coefficient $a_{i,j}$ with $i\not=j$.
Then, the multilinearity of the determinant implies that
\[\det\big(t(t1_p-\alpha a)-bc\big)=f(t^2)+\alpha \, t \, g(t^2)\]
for some polynomials $f$ and $g$.
For $\alpha=1$, as the matrix $x(1)=x$ is nilpotent, we must have
\[t^{p+q}=t^{q-p}(f(t^2)+tg(t^2)).\]
This equality forces $f(t^2)=t^{2p}$ and $tg(t^2)=0$.
Thereby,
\[\det(t1_{p+q}-x(\alpha))=t^{p+q}\ \mbox{ for all $\alpha\in\mathbb{C}$}.\]
Thus $x(\alpha)$ is nilpotent for all $\alpha$.
In particular, $x^{-\theta}=x(0)$ is nilpotent.

Second, let us assume that we are in Case (b), i.e., $P$ is maximal.
In that case (as noted in Section \ref{section-3-2}) the Richardson orbit of $P$ consists of two-step nilpotent matrices, hence 
$x^2=0$.  
If we express $ x $ as in \eqref{eq:x-is-ABCD}, we get 
\begin{equation}
\label{14}
x^2 = \begin{pmatrix}
a^2 + bc & ab + bd 
\\
ca + dc & cb + d^2
\end{pmatrix}
= 0.
\end{equation}
Whence 
\begin{equation}\label{eq:square-of-x-minus-theta}
(x^{-\theta})^2 
= \begin{pmatrix}
bc & 0 
\\
0 & cb
\end{pmatrix}
= - \begin{pmatrix}
a^2 &  0 
\\
 0  & d^2
\end{pmatrix}=-(x^\theta)^2.
\end{equation}
Since $x^{\theta}$ is nilpotent, we conclude from (\ref{eq:square-of-x-minus-theta}) that $x^{-\theta}$ is also nilpotent.
\end{proof}

\begin{remark}
It follows from (\ref{14}) and (\ref{eq:square-of-x-minus-theta}) that
if $x$ and $x^\theta$ are respectively two-step nilpotent and $k$-step nilpotent, then 
$x^{-\theta}$ is at most $(k+1)$-step nilpotent when $k$ is odd and  at most $k$-step nilpotent when $k$ is even.
\end{remark}

\begin{corollary}
\label{corollary-4-3}
Under the assumptions of Proposition \ref{proposition-4-2},
the map $\pi:\mathcal{Y}\to K/Q\times\mathcal{N}(\mathfrak{s})$, $(\mathfrak{p}_1,\mathfrak{q}_1,x)\mapsto (\mathfrak{q}_1,x^{-\theta})$ is well defined.
\end{corollary}

\begin{proof}
By (\ref{eq:conormal-variety-Y}) we have $x\in \mathfrak{n}_{\mathfrak{p}_1}\subset\overline{\mathcal{O}_P^G}$ and $x^\theta\in\mathfrak{n}_{\mathfrak{q}_1}\subset\overline{\mathcal{O}_Q^K}$ whenever $(\mathfrak{p}_1,\mathfrak{q}_1,x)\in\mathcal{Y}$. The conclusion follows from Proposition \ref{proposition-4-2}.
\end{proof}

\subsection{Properties of the map $\pi$ for $Q$ mirabolic}
In what follows we restrict our attention to pairs $(P,Q)$ of standard parabolic subgroups of $G$ and $K$ such that $Q$ is mirabolic. Specifically we take
\begin{equation}
\label{Q-8}
Q=\{g\in K:g(L_1)\subset L_1\}
\end{equation}
for the line $L_1=\langle (1,0,\ldots,0)\rangle_\mathbb{C}\subset \mathbb{C}^p=V_1$. Thus $K/Q$ is the projective space $\mathbb{P}(V_1)$.

Let $P=P_{\underline{d}}$ be a standard parabolic subgroup of $G$ corresponding to the
composition $\underline{d}=(d_1,\ldots,d_r)$.
Thus, the double flag variety $\Xfv=G/P\times K/Q$ interprets as the product
\[\Xfv=\mathcal{F}_{\underline{d}}\times \mathbb{P}(V_1)\]
whose elements are pairs $(W_\bullet,L)$ where $L\subset V_1$ is a line and
$W_\bullet=(W_0=0\subset W_1\subset \ldots\subset W_r=V)$ is a partial flag with $\dim W_i/W_{i-1}=d_i$.
The conormal variety $\mathcal{Y}$ is then
\begin{eqnarray}
\label{17}
\mathcal{Y}=\{(W_\bullet,L,x)\in\mathcal{F}_{\underline{d}}\times\mathbb{P}(V_1)\times\mathcal{L}(V):\mathrm{Im}\,x^\theta\subset L\subset\ker x^\theta \\ \mbox{and }x(W_i)\subset W_{i-1}\ \forall i\geq 1\}. \nonumber
\end{eqnarray}
By Corollary \ref{corollary-4-3}, the map $\pi:\mathcal{Y}\to\mathbb{P}(V_1)\times\mathcal{N}(\mathfrak{s})$, $(W_\bullet,L,x)\mapsto (L,x^{-\theta})$ is well defined.
Some technical properties of this map are described in the next lemma.

\begin{lemma}\label{lemma-4-4}
Let $(W_\bullet,L,x)\in \mathcal{Y}$. Write $L=\langle v\rangle_\mathbb{C}$ with $v\in M_{p,1}(\mathbb{C})\cong V_1$, $v\not=0$. \\
{\rm (a)} We can write
\[x=\left(\begin{array}{cc}
\eta & a \\ b & 0
\end{array}\right)\]
where $\eta=v\cdot{}^tu$ for some $u\in M_{p,1}(\mathbb{C})$ such that ${}^tu\cdot v=0$, and with $a\in M_{p,q}(\mathbb{C})$, $b\in M_{q,p}(\mathbb{C})$. 
\\
{\rm (b)} We have $\rk\,x^{-\theta}\in\{\rk x,\rk x-1\}$.
Moreover, the equality $\rk x=\rk x^{-\theta}$ holds if and only if 
$(v\in\mathrm{Im}\,a$ or $u\in\mathrm{Im}\,{}^tb)$. \\
{\rm (c)} If $x^2=0$, then $(x^{-\theta})^2=0$.
\end{lemma}

\begin{proof}
(a) The form of the matrix is a consequence of the property
$\mathrm{Im}\,x^\theta\subset L \subset \ker x^\theta$. \\
(b) For all $\alpha\in\mathbb{C}^*$, the matrix
\[x(\alpha)=\left(\begin{array}{cc}
\alpha\eta & a \\ b & 0
\end{array}\right)\]
has same rank as $x$. The lower semicontinuity of the rank yields
$\rk x^{-\theta}=\rk x(0)\leq \rk x$.

Assume that $v\in\mathrm{Im}\,a$. Then, $\mathrm{Im}\,\eta\subset\mathrm{Im}\,a$ and we get
\[\mathrm{Im}\,x\subset \mathrm{Im}\,b\oplus(\mathrm{Im}\,a+\mathrm{Im}\,\eta)=\mathrm{Im}\,b\oplus\mathrm{Im}\,a=\mathrm{Im}\,x^{-\theta}.\]
Since $\rk x^{-\theta}\leq\rk x$, the equality $\rk x^{-\theta}=\rk x$ must hold.

Assume that $u\in\mathrm{Im}\,{}^tb$. This implies that $\ker b\subset \ker \eta$. In this case, one has
\[\ker x^{-\theta}=\ker b\oplus \ker a=(\ker b\cap\ker\eta)\oplus\ker a\subset \ker x.\]
Since $\rk x^{-\theta}\leq \rk x$, we get $\ker x^{-\theta}=\ker x$, whence $\rk x^{-\theta}=\rk x$.

Finally, suppose that $v\notin\mathrm{Im}\,a$ and $u\notin\mathrm{Im}\,{}^tb$.  The latter condition yields an element $w\in\ker b\setminus\ker\eta$. On the one hand, the fact that $w\notin \ker \eta$ implies that $\eta w=\alpha v$ for some $\alpha\in\mathbb{C}^*$. On the other hand, since $w\in\ker b$, we get
\[\alpha v=\eta w=xw\in\mathrm{Im}\,x,\quad\mbox{so}\ \ \mathrm{Im}\,\eta\subset\mathrm{Im}\,x.\]
The latter inclusion implies $\mathrm{Im}\,b\subset \mathrm{Im}\,x$.
Moreover, we always have $\mathrm{Im}\,a\subset\mathrm{Im}\,x$.
Thereby,
\[\mathrm{Im}\,x=(\mathrm{Im}\,a+\mathrm{Im}\,\eta)\oplus\mathrm{Im}\,b=\mathrm{Im}\,a\oplus\mathrm{Im}\,\eta\oplus\mathrm{Im}\,b=\mathrm{Im}\,\eta\oplus\mathrm{Im}\,x^{-\theta},\]
where the second equality follows from the assumption that $v\notin\mathrm{Im}\,a$. This yields
\[\rk x=\rk x^{-\theta}+1.\]
{\rm (c)} Note that
\[x^2=\left(\begin{array}{cc}
\eta^2+ab & \eta a \\ b\eta & ba
\end{array}\right).\]
Since $\eta^2=0$,
the assumption that $x^2=0$ yields $ab=0$ and
$ba=0$, which guarantees that $(x^{-\theta})^2=0$.
\end{proof}

\subsection{Proof of Proposition \ref{proposition-2-8}}\label{subsection:proof-prop-2-8}

Here we consider the situation of Proposition \ref{proposition-2-8}: $Q$ is a mirabolic subgroup of $K$ as in (\ref{Q-8}) and $P=P_{\underline{d}}\subset G$ is the maximal parabolic subgroup corresponding to the composition $\underline{d}=(k,n-k)$.
As noted in Section \ref{section-3-2}, the Richardson
orbit corresponding to $P$ is $\mathcal{O}_P^G=\mathcal{O}^G_{(\max\{k,n-k\},\min\{k,n-k\})^*}$, which means that every $x\in\overline{\mathcal{O}_P^G}$ satisfies $x^2=0$ and $\rk x\leq\min\{k,n-k\}$. Corollary \ref{corollary-4-3}, Lemma \ref{lemma-4-4}\,{\rm (b)} and {\rm (c)}, and the definition of $\mathfrak{E}=\mathbb{P}(V_1)\times\mathcal{N}(\mathfrak{s})_2^k$ in Notation \ref{notation-1-5}\,{\rm (c)} imply that the map
\[\pi:\mathcal{Y}\to\mathfrak{E},\ (W,L,x)\mapsto (L,x^{-\theta})\]
is well defined.

For every $(L,z)\in\mathfrak{E}$, we have in particular $z\in\overline{\mathcal{O}_P^G}$. This inclusion guarantees that the Spaltenstein variety $\mathcal{F}_{z,(k,n-k)}\subset\Gr_k(V)$ is nonempty (see Section \ref{section-3-3}). Any $W\in\mathcal{F}_{z,(k,n-k)}$ satisfies $(W,L,z)\in\mathcal{Y}$ and $\pi((W,L,z))=(L,z)$. Therefore, the map $\pi$ is surjective onto $\mathfrak{E}$.
This completes the proof of Proposition \ref{proposition-2-8}.

\section{Proof of Theorem~\ref{theorem-2}: Orbit correspondence}

In Section \ref{section-5-1} we show an abstract correspondence between the $K$-orbits of the exotic nilpotent cone and the components of the conormal variety associated to the double flag variety $G/P\times K/Q$,
in the case where $Q\subset K$ is mirabolic and for $P\subset G$ arbitrary.
In Sections \ref{section-5-2}--\ref{section-5-4} we assume in addition $P$ maximal, 
so that we recover the situation of the exotic Grassmannian $\Gr_k(V)\times\mathbb{P}(V_1)$, 
and we prove Theorem~\ref{theorem-2} by using the abstract result of Section \ref{section-5-1}.

\subsection{Abstract correspondence}\label{section-5-1}

All along the section we take $Q$ mirabolic, specifically we assume that $Q\subset K$ is the stabilizer of a line of $V_1$, so that $K/Q=\mathbb{P}(V_1)$. In this subsection we let
$P=P_{\underline{d}}\subset G$ be an arbitrary (standard) parabolic subgroup, corresponding to a composition $\underline{d}=(d_1,\ldots,d_r)$.
As in Section \ref{section-3-2} we denote by $\mathcal{O}_P^G\subset\mathcal{N}(\mathfrak{g})$ the Richardson orbit corresponding to $P$.
The conormal variety $\mathcal{Y}$ is as in (\ref{17}), and it is equidimensional (this follows from Proposition \ref{proposition-1}; see Section \ref{section-1-3}).

Set 
\[\mathfrak{n}_L=\{z\in\mathfrak{k}=\mathcal{L}(V_1)\times\mathcal{L}(V_2):\mathrm{Im}\,z\subset L\subset\ker z\}\] 
(this is the nilradical of $\Lie(Q)$ in $\mathfrak{k}$). Hence
\[
\mathcal{Y}=\{(W_\bullet,L,x)\in\mathcal{F}_{\underline{d}}\times\mathbb{P}(V_1)\times\mathcal{N}(\mathfrak{g}):
x^\theta\in\mathfrak{n}_L\mbox{ and $x(W_i)\subset W_{i-1}$ for all $i\geq 1$}\}.
\]
By Corollary \ref{corollary-4-3} the
map 
\[\pi:\mathcal{Y}\to \mathbb{P}(V_1)\times\mathcal{N}(\mathfrak{s}),\ (W_\bullet,L,x)\mapsto (L,x^{-\theta})\]
is well defined. We write $\pi=\psi\circ\phi$ where
\[\phi:\mathcal{Y}\to\mathbb{P}(V_1)\times\mathcal{N}(\mathfrak{g}),\ (W_\bullet,L,x)\mapsto(L,x)\]
and
\[\psi:\mathcal{Z}:=\phi(\mathcal{Y})=\{(L,x)\in\mathbb{P}(V_1)\times\overline{\mathcal{O}_P^G}\overline{}:x^\theta\in\mathfrak{n}_L\}\to\mathbb{P}(V_1)\times\mathcal{N}(\mathfrak{s}),\ (L,x)\mapsto (L,x^{-\theta}).\]
Furthermore we set $\mathfrak{E}=\pi(\mathcal{Y})=\psi(\mathcal{Z})$. 

It is easy to see that $K$ acts diagonally on $\mathcal{Y}$, $\mathcal{Z}$, and $\mathfrak{E}$, 
and the maps $\phi$ and $\psi$ are $K$-equivariant.
By Proposition \ref{proposition-2-1}, $\mathfrak{E}$ has a finite decomposition into $K$-orbits
\[\mathfrak{E}=\bigsqcup_{i=1}^m\mathfrak{O}_i^K.\]
For every $i\in\{1,\ldots,m\}$ and 
every partition $\lambda\vdash n$ for which 
$ \mathcal{O}_{\lambda}^G \subset \overline{\mathcal{O}_P^G}$, we denote
\[\mathcal{Z}_i=\psi^{-1}(\mathfrak{O}_i^K)\quad\mbox{and}\quad
\mathcal{Z}_i^\lambda=\{(L,x)\in\mathcal{Z}_i:x\in\mathcal{O}_{\lambda}^G \}.\]
Thus every subset $\mathcal{Z}_i$ or $\mathcal{Z}_i^\lambda$ is a locally closed subvariety of $\mathcal{Z}$, and every inverse image $\phi^{-1}(\mathcal{Z}_i)$ or $\phi^{-1}(\mathcal{Z}_i^\lambda)$ is a locally closed subvariety of $\mathcal{Y}$.
Note that
\begin{equation}
\label{18-new}
\mathcal{Z}_i^\lambda=\{h(L_i,x):h\in K,\ x\in\mathcal{O}_\lambda^G\cap(z_i+\mathfrak{n}_{L_i})\}\quad\mbox{whenever $(L_i,z_i)\in\mathfrak{O}_i^K$.}
\end{equation}
Note also that the restriction 
$\phi|_{\phi^{-1}(\mathcal{Z}_i^\lambda)}:\phi^{-1}(\mathcal{Z}_i^\lambda)\to \mathcal{Z}_i^\lambda$ 
is a fibration whose fiber is the Spaltenstein variety $\mathcal{F}_{\lambda,\underline{d}}$, thus we have $\dim\phi^{-1}(\mathcal{Z}_i^\lambda)=\dim\mathcal{Z}_i^\lambda+\dim\mathcal{F}_{\lambda,\underline{d}}$.

\begin{definition}
\label{definition-good-pair}
For a $K$-orbit $\mathfrak{O}_i^K\subset\mathfrak{E}$ and a partition $\lambda\vdash n$ with $\mathcal{O}_\lambda^G\subset\overline{\mathcal{O}_P^G}$,
we say that the pair $(\mathfrak{O}_i^K,\lambda)$ is {\em good} if $\mathcal{Z}_i^\lambda$ is nonempty and the equality
\[
\dim\mathcal{Z}_i^\lambda+\dim\mathcal{F}_{\lambda,\underline{d}}=\dim\mathcal{Y}
\]
holds. 
\end{definition}

The next statement is a general construction of the irreducible components of $\mathcal{Y}$. We use it later in a special case (for $P$ maximal) in the proof of Theorem \ref{theorem-2}.

\begin{proposition}
\label{proposition-5-1}
{\rm (a)}
For every irreducible component $\mathcal{C}\subset\mathcal{Y}$, there is a unique good pair $(\mathfrak{O}_i^K,\lambda)$ such that $\mathcal{C}\subset\overline{\phi^{-1}(\mathcal{Z}_i^\lambda)}$. \\
{\rm (b)}
Conversely, if $(\mathfrak{O}_i^K,\lambda)$ is a good pair, then every component of $\overline{\phi^{-1}(\mathcal{Z}_i^\lambda)}$ of maximal dimension is a component of $\mathcal{Y}$.
In addition, there is a bijection 
\[\mathrm{Irr}_{\mathrm{max}}(\overline{\phi^{-1}(\mathcal{Z}_i^\lambda)})\cong \mathrm{Irr}_{\mathrm{max}}(\mathcal{Z}_i^\lambda)\times\mathrm{Irr}(\mathcal{F}_{\lambda,\underline{d}}),\]
where $\mathrm{Irr}_{\mathrm{max}}(Z)$ stands for the set of irreducible components of maximal dimension in  $Z$. \\
{\rm (c)} For a good pair $(\mathfrak{O}_i^K,\lambda)$, choose a point $(L_i,z_i)\in\mathfrak{O}_i^K$.
If the variety $\mathcal{O}_\lambda^G\cap(z_i+\mathfrak{n}_{L_i})$ is irreducible,
then $\mathcal{Z}_i^\lambda$ is irreducible,
$\overline{\phi^{-1}(\mathcal{Z}_i^\lambda)}$ is equidimensional,
and there is a bijection
$\mathrm{Irr}(\overline{\phi^{-1}(\mathcal{Z}_i^\lambda)})\cong \mathrm{Irr}(\mathcal{F}_{\lambda,\underline{d}})$.\\
\end{proposition}

\begin{proof}
{\rm (a)} We have
$\mathcal{Z}=\phi(\mathcal{Y})=\bigsqcup_{(\mathfrak{O}_i^K,\lambda)}\mathcal{Z}_i^\lambda$
(where the union is over all pairs $(\mathfrak{O}_i^K,\lambda)$, not necessarily good) hence
\begin{equation}
\label{19}
\mathcal{Y}=\bigsqcup_{(\mathfrak{O}_i^K,\lambda)}\phi^{-1}(\mathcal{Z}_i^\lambda)=\bigcup_{(\mathfrak{O}_i^K,\lambda)}\overline{\phi^{-1}(\mathcal{Z}_i^\lambda)}.
\end{equation}
Let a component $\mathcal{C}\subset\mathcal{Y}$. 
By (\ref{19}), there is a pair $(\mathfrak{O}_i^K,\lambda)$ such that $\mathcal{C}\subset\overline{\phi^{-1}(\mathcal{Z}_i^\lambda)}$. Since $\mathcal{Y}$ is equidimensional, we know that $\dim\mathcal{C}=\dim\mathcal{Y}$.
Whence
\[\dim\mathcal{Y}=\dim\mathcal{C}\leq \dim\phi^{-1}(\mathcal{Z}_i^\lambda)=\dim\mathcal{Z}_{i}^\lambda+\dim\mathcal{F}_{\lambda,\underline{d}}\leq\dim\mathcal{Y},\]
which forces the pair $(\mathfrak{O}_i^K,\lambda)$ to be good. 
Note that the pair $(\mathfrak{O}_i^K,\lambda)$ is necessarily unique, since 
$ \phi^{-1}(\mathcal{Z}_{i}^\lambda)\cap\mathcal{C}$ must be open and dense in the irreducible component $ \mathcal{C} $.
This completes the proof of part {\rm (a)}. 

\noindent
{\rm (b)}
Let $(L_i,z_i)\in\mathfrak{O}_i^K$. Denote $H=\{h\in K:h(L_i,z_i)=(L_i,z_i)\}$, which is a connected subgroup of $K$
(it is connected since it is open in the linear space $\{u\in\mathcal{L}(V_1)\times\mathcal{L}(V_2): u(L_i)\subset L_i,\ [u,z_i]=0\}$). We choose a partition into $H$-stable locally closed subsets
\[\mathcal{O}_\lambda^G\cap(z_i+\mathfrak{n}_{L_i})=C_1\sqcup\ldots\sqcup C_\ell\]
such that the irreducible components of $\mathcal{O}_\lambda^G\cap(z_i+\mathfrak{n}_{L_i})$ are exactly the closures of the subsets $C_j$ for $j\in\{1,\ldots,\ell\}$.
In view of (\ref{18-new}), we derive the partition
\[\mathcal{Z}_i^\lambda=\mathcal{D}_1\sqcup\ldots\sqcup\mathcal{D}_\ell\quad\mbox{where}\quad
\mathcal{D}_j:=\{h(L_i,x):h\in K,\ x\in C_j\}\ \mbox{ for all $j$},\]
the pairwise disjointness of the $\mathcal{D}_j$'s being shown as follows: the condition $\mathcal{D}_j\cap\mathcal{D}_{j'}\not=\emptyset$ yields $h\in K$, $x\in C_j$, and $x'\in C_{j'}$ such that $h(L_i,x)=(L_i,x')$; whence $h\in H$ (since $h z_i=(h x)^{-\theta}=x'^{-\theta}=z_i$), $x'=h x\in C_j\cap C_{j'}$ (since $C_j$ is $H$-stable), and so $j=j'$.
Each subset $\mathcal{D}_j$ is a constructible subset of $\mathcal{Z}_i^\lambda$ whose closure is irreducible.
Whence a clearly well-defined and bijective map
\begin{equation}
\label{19-bis}
\big\{j\in\{1,\ldots,\ell\}:\dim\mathcal{D}_j=\dim\mathcal{Z}_i^\lambda\big\}\to\mathrm{Irr}_{\mathrm{max}}(\mathcal{Z}_i^\lambda),\ j\mapsto\overline{\mathcal{D}_j}\cap\mathcal{Z}_i^\lambda.
\end{equation}

For $j\in\{1,\ldots,\ell\}$, fix an element $x_j\in C_j$ and let
\[G_j:=\{g\in G:gx_j\in C_j\}.\]
Note that $G_j$ is the inverse image of $C_j$ by the map $G\to\mathcal{O}_{\lambda}^G$, $g\mapsto gx_j$. The latter map is a locally trivial fiber bundle of (connected) fiber $Z_G(x_j):=\{g\in G:gx_j=x_j\}$ (see \cite[\S2.6]{Popov.Vinberg.1994}). Hence $G_j$ is an irreducible locally closed subset of $G$.
In addition we consider the Spaltenstein variety $\mathcal{F}_{x_j,\underline{d}}$ and we can find a partition
\[\mathcal{F}_{x_j,\underline{d}}=\mathcal{G}_1\sqcup\ldots\sqcup\mathcal{G}_s\]
into $Z_G(x_j)$-stable locally closed subsets such that the closures $\overline{\mathcal{G}_t}$
are exactly the components of $\mathcal{F}_{x_j,\underline{d}}$. 
For every pair $(j,t)\in\{1,\ldots,\ell\}\times\{1,\ldots,s\}$, we set
\[\mathcal{Y}_{j,t}=\{h(gW_\bullet,L_i,gx_j):h\in K,\ g\in G_j,\ W_\bullet\in\mathcal{G}_t\}.\]
Then $\mathcal{Y}_{j,t}$ is a constructible subset of $\phi^{-1}(\mathcal{D}_j)$ whose closure is irreducible. Note that
\[\phi^{-1}(\mathcal{D}_j)=\mathcal{Y}_{j,1}\sqcup\ldots\sqcup\mathcal{Y}_{j,\ell}\]
where the disjointness of the union is verified as follows: if $\mathcal{Y}_{j,t}\cap\mathcal{Y}_{j,t'}\not=\emptyset$, then we find in particular $g\in G$, $W_\bullet\in\mathcal{G}_t$, and $W'_\bullet\in\mathcal{G}_{t'}$ such that $(gW_\bullet,gx_j)=(W'_\bullet,x_j)$; hence $g\in Z_G(x_j)$, $W'_\bullet=gW_\bullet\in\mathcal{G}_t\cap\mathcal{G}_{t'}$ (since $\mathcal{G}_t$ is $Z_G(x_j)$-stable), and finally
$t=t'$. We obtain
\begin{equation}
\label{19-ter}
\overline{\phi^{-1}(\mathcal{D}_j)}=\overline{\mathcal{Y}_{j,1}}\cup\ldots\cup\overline{\mathcal{Y}_{j,\ell}}\quad\mbox{and}\quad
\overline{\phi^{-1}(\mathcal{Z}_i^\lambda)}=\bigcup_{j=1}^\ell\bigcup_{t=1}^s\overline{\mathcal{Y}_{j,t}}.
\end{equation}
The $\overline{\mathcal{Y}_{j,t}}$'s are irreducible closed subsets of $\overline{\phi^{-1}(\mathcal{Z}_i^\lambda)}$.
Moreover the restriction
$\phi|_{\mathcal{Y}_{j,t}}:\mathcal{Y}_{j,t}\to\mathcal{D}_j$
has constant fiber (isomorphic to) $\mathcal{G}_t$.
Since $(\mathfrak{O}_i^K,\lambda)$ is a good pair, this yields
\[\dim\overline{\mathcal{Y}_{j,t}}=\dim \mathcal{D}_j+\dim\mathcal{G}_t=\dim\mathcal{D}_j+\dim\mathcal{F}_{\lambda,\underline{d}}
\leq\dim\mathcal{Z}_i^\lambda+\dim\mathcal{F}_{\lambda,\underline{d}}=\dim \mathcal{Y}.\]
Therefore, $\overline{\mathcal{Y}_{j,t}}$ is an irreducible component of $\mathcal{Y}$
(or, equivalently, a component of $\overline{\phi^{-1}(\mathcal{Z}_i^\lambda)}$ of maximal dimension) if and only if $\dim\mathcal{D}_j=\dim\mathcal{Z}_i^\lambda$.
Whence the map
\begin{equation}
\label{bijective-map-2}
\big\{j\in\{1,\ldots,\ell\}:\dim\mathcal{D}_j=\dim\mathcal{Z}_i^\lambda\big\}\times\{1,\ldots,s\}\to\mathrm{Irr}_{\mathrm{max}}(\overline{\phi^{-1}(\mathcal{Z}_i^\lambda})),\ (j,t)\mapsto\overline{\mathcal{Y}_{j,t}}
\end{equation}
is well defined. This map is surjective (by (\ref{19-ter}))
and injective (since $\mathcal{Y}_{j,t}$ and $\mathcal{Y}_{j',t'}$ are disjoint whenever $(j,t)\not=(j',t')$ and contain dense open subsets of $\overline{\mathcal{Y}_{j,t}}$ and $\overline{\mathcal{Y}_{j',t'}}$, respectively).
Comparing the bijections of (\ref{19-bis}) and (\ref{bijective-map-2}) achieves the proof of part {\rm (b)}.

\noindent
{\rm (c)} In the case where $\mathcal{O}_\lambda^G\cap(z_i+\mathfrak{n}_{L_i})$ is irreducible, we have $\ell=1$. Hence (\ref{19-ter}) becomes
\[\overline{\phi^{-1}(\mathcal{Z}_i^\lambda)}=\overline{\mathcal{Y}_{1,1}}\cup\ldots\cup \overline{\mathcal{Y}_{1,s}}.\]
Since $\overline{\mathcal{Y}_{1,1}},\ldots,\overline{\mathcal{Y}_{1,s}}$ are all of the same dimension, we get that $\overline{\phi^{-1}(\mathcal{Z}_i^\lambda)}$ is equidimensional and $\mathrm{Irr}(\overline{\phi^{-1}(\mathcal{Z}_i^\lambda)})=\{\overline{\mathcal{Y}_{1,1}},\ldots,\overline{\mathcal{Y}_{1,s}}\}\cong\mathrm{Irr}(\mathcal{F}_{\lambda,\underline{d}})$.
\end{proof}

\subsection{Outline of the proof of Theorem \ref{theorem-2}}\label{section-5-2}

In the rest of the section, we consider the situation in Theorem~\ref{theorem-2}; 
namely, in addition to assuming $Q$ mirabolic (as in the previous subsection), we assume that the parabolic subgroup $P\subset G$ is maximal, of the form $P=P_{(k,n-k)}$. 
In this special case, we consider the factorization of $\pi:\mathcal{Y}\to\mathfrak{E}$ through the two maps
\[\mathcal{Y}\stackrel{\phi}{\longrightarrow}\mathcal{Z}\stackrel{\psi}{\longrightarrow}\mathfrak{E},\]
as explained in the previous subsection.  
The exotic nullcone $\mathfrak{E}$, 
which is the image $ \pi(\mathcal{Y}) $ by definition, 
coincides with the one introduced in Notation \ref{notation-1-5}\,{\rm (c)}, 
i.e., $\mathfrak{E}=\mathbb{P}(V_1)\times\mathcal{N}(\mathfrak{s})_2^k$.

By Corollary \ref{corollary-2-6} the $K$-orbits of $\mathfrak{E}$ are of the form $\mathfrak{O}_{((r,s),\mu)}^K$ and they are parametrized by the pairs $((r,s),\mu)\in\Pi_2^k$.
In particular for $(L,z)\in\mathfrak{O}_{((r,s),\mu)}^K$ we have $z^2=0$ and $\rk z=r+s$.
Set $\mathcal{Z}_{((r,s),\mu)}=\psi^{-1}(\mathfrak{O}_{((r,s),\mu)}^K)$, 
and for every partition $\lambda\vdash n$,  
we denote 
\begin{equation*}
\mathcal{Z}^\lambda_{((r,s),\mu)}=\{(L,x)\in\mathcal{Z}_{((r,s),\mu)}:x\in\mathcal{O}_{\lambda}^G\} .
\end{equation*}
We restrict our attention to partitions of the form $\lambda=(n-\ell,\ell)^*$ with $\ell\leq\min\{k,n-k\}$ (otherwise $\mathcal{O}_{\lambda}^G$ is not contained in the closure of the Richardson orbit $\mathcal{O}_P^G$, thence $\mathcal{Z}^\lambda_{((r,s),\mu)}$ is empty).

In Section \ref{section-5-3}, we will show the following property:
\begin{equation}
\label{20}
\mbox{For every $(L,z)\in\mathfrak{E}$, the set $\mathcal{O}_{\lambda}^G\cap(z+\mathfrak{n}_L)$ is always empty or irreducible.}
\end{equation}
This fact together with (\ref{18-new}) implies that $\mathcal{Z}_{((r,s),\mu)}^\lambda$ is empty or irreducible for every pair $(\mathfrak{O}_{((r,s),\mu)}^K,\lambda)$. Moreover, if $\mathcal{Z}_{((r,s),\mu)}^\lambda$ is nonempty, then Proposition \ref{proposition-2} implies that
\begin{equation}
\label{21}
\mbox{the Spaltenstein variety $\mathcal{F}_{\lambda,(k,n-k)}$ is irreducible}.
\end{equation}
Comparing (\ref{20}) and (\ref{21}) with Proposition \ref{proposition-5-1}\,{\rm (c)}, we obtain:
\begin{equation}
\label{22}
\parbox{14cm}{The irreducible components of $\mathcal{Y}$ are exactly the subsets $\overline{\phi^{-1}(\mathcal{Z}_{((r,s),\mu)}^\lambda)}$ for the good pairs $(\mathfrak{O}_{((r,s),\mu)}^K,\lambda)$.}
\end{equation}
Therefore, in order to complete the proof of Theorem \ref{theorem-2}, it suffices to show the following facts:
\begin{equation}
\label{23}
\parbox{14cm}{If $((r,s),\mu)$ is of type (I), $\mbox{(II)}^0$, or (III), then there is a unique partition $\lambda\vdash n$ such that the pair $(\mathfrak{O}_{((r,s),\mu)}^K,\lambda)$ is good.}
\end{equation}
\begin{equation}
\label{24}
\parbox{14cm}{If $((r,s),\mu)$ is of type $\mbox{(II)}^*$, then there are exactly two partitions $\lambda,\lambda'\vdash n$ such that the pairs $(\mathfrak{O}_{((r,s),\mu)}^K,\lambda)$ and $(\mathfrak{O}_{((r,s),\mu)}^K,\lambda')$ are good.}
\end{equation}
The proof of (\ref{23}) and (\ref{24}) will be carried out in Section \ref{section-5-4}.

\subsection{Fibers of the map $\psi$}
\label{section-5-3}

The purpose of this subsection is to show property (\ref{20}). So we fix a pair 
$(L,z)\in\mathfrak{E}=\mathbb{P}(V_1)\times\mathcal{N}(\mathfrak{s})_2^k$.
Write
\begin{equation}\label{eq:L.and.z}
L=\langle v\rangle_{\mathbb{C}}\quad\mbox{and}\quad z=\left(\begin{array}{cc}
0 & a \\ b & 0
\end{array}\right)
\end{equation}
with a nonzero $v\in M_{p,1}(\mathbb{C})\cong V_1$, $a\in M_{p,q}(\mathbb{C})$, and $b\in M_{q,p}(\mathbb{C})$. Our goal is to study the fiber $\psi^{-1}((L,z))$
and the intersections $\mathcal{O}_{\lambda}^G\cap (z+\mathfrak{n}_L)$, and the next lemma is a key.

\begin{lemma}
\label{lemma-3}
Let $ z \in \mathcal{N}(\mathfrak{s})_2^k $ be as in \eqref{eq:L.and.z}, and set 
\[\Phi(z):=\{x\in\mathcal{N}(\mathfrak{g}):x^2=0,\ x^{-\theta}=z,\ x^\theta\in\mathfrak{n}_L \}.\]
{\rm (a)} If $bv\not=0$, then $\Phi(z)=\{z\}$. \\
{\rm (b)} If $bv=0$, then
\[\Phi(z)=\left\{x(u)=\left(\begin{array}{cc}
v\cdot{}^tu & a \\ b & 0
\end{array}\right): u\in\ker{}^ta,\ {}^tu\cdot v=0  \right\}.\]
Moreover, in this case, denoting $\ell:=\rk z=\rk a+\rk b\in\{0,\ldots,\min\{p,q,k,n-k\}\}$,
\begin{itemize}
\item[\rm (i)] if $v\in\mathrm{Im}\,a$, then $\Phi(z)\subset\mathcal{O}^G_{(n-\ell,\ell)^*}$ and $\dim \Phi(z)=p-\rk a$.
\item[\rm (ii)] If $v\notin\mathrm{Im}\,a$, then
\[\Phi(z)=\big(\Phi(z)\cap\mathcal{O}^G_{(n-\ell,\ell)^*}\big)\cup\big(\Phi(z)\cap\mathcal{O}^G_{(n-\ell-1,\ell+1)^*}\big)\]
and we have
\begin{eqnarray*}
 & \Phi(z)\cap\mathcal{O}^G_{(n-\ell,\ell)^*}=\{x(u):u\in\mathrm{Im}\,{}^tb\} ,  \\
 & \dim \bigl( \Phi(z)\cap\mathcal{O}^G_{(n-\ell,\ell)^*} \bigr) =\rk b,
\end{eqnarray*}
and
\[\Phi(z)\cap\mathcal{O}^G_{(n-\ell-1,\ell+1)^*}=\{x(u):u\in\ker{}^ta,\ {}^tu\cdot v=0,\ u\notin\mathrm{Im}\,{}^tb\},\]
which is nonempty if and only if $\ell<p-1$, in which case we have
\[\dim \bigl( \Phi(z)\cap\mathcal{O}^G_{(n-\ell-1,\ell+1)^*} \bigr) =p-\rk a-1.\]
\end{itemize}
{\rm (c)} In particular, in all cases, for every partition $\lambda\vdash n$, the intersection $\Phi(z)\cap\mathcal{O}^G_\lambda$ is irreducible or empty.
\end{lemma}

\begin{proof}
By assumption, we have $z^2=0$, hence 
\begin{equation}
\label{25} ab=0\quad\mbox{and}\quad ba=0.
\end{equation}
By definition, an element $x$ belongs to $\Phi(z)$ if and only if it is of the form
\[x=\left(\begin{array}{cc}
\eta & a \\ b & 0
\end{array}\right)\]
and satisfies $x^2=0$ and $\mathrm{Im}\,\eta\subset L=\langle v\rangle_\mathbb{C}\subset \ker \eta$.
The latter condition forces that $\eta=v\cdot{}^tu$ with $u\in M_{p,1}(\mathbb{C})$ such that ${}^tu\cdot v=0$. 
Thus, $\eta^2=0$. Note that
\[x^2=\left(\begin{array}{cc}
\eta^2+ab & \eta a \\ b\eta & ba
\end{array}\right)=\left(\begin{array}{cc}
0 & \eta a \\ b\eta & 0
\end{array}\right)\]
(by (\ref{25})).
Thereby, the condition $x^2=0$ is equivalent to
\[\eta a=0\quad\mbox{and}\quad b\eta=0,\]
which is equivalent to
\[u\in\ker{}^ta\quad\mbox{and}\quad(v\in\ker b\quad\mbox{or}\quad u=0).\]
The lemma easily follows from these observations and from Lemma \ref{lemma-4-4}.
\end{proof}

The fiber $\psi^{-1}((L,z))$ is determined in the following proposition.

\begin{proposition}
\label{proposition-4}
Let $(L,z)=(\langle v\rangle_\mathbb{C},\left(\begin{smallmatrix}
0 & a \\ b & 0
\end{smallmatrix}\right))\in \mathfrak{E}=\mathbb{P}(V_1)\times\mathcal{N}(\mathfrak{s})_2^k$ as before and denote $\ell=\rk z\in\{0,\ldots,\min\{p,q,k,n-k\}\}$. \\
{\rm (a)} Assume that $\ell=\min\{k,n-k\}$ or $(\ell<\min\{k,n-k\}$ and $(v\notin\ker b$ or $v\in\mathrm{Im}\,a$ or $\ell\geq p-1))$. 
Then $\psi^{-1}((L,z))\subset \mathbb{P}(V_1)\times\mathcal{O}^G_{(n-\ell,\ell)^*}$. Moreover, $\psi^{-1}((L,z))$ is an affine space of dimension
\[\left\{\begin{array}{ll}
0 & \mbox{if $v\notin\ker b$,} \\
p-\rk a & \mbox{if $v\in\mathrm{Im}\,a(\subset\ker b)$,} \\
\rk b & \mbox{if $v\in\ker b\setminus\mathrm{Im}\,a$.}
\end{array}\right.\]
{\rm (b)} Let $\ell<\min\{k,n-k\}$ and assume that $v\in\ker b\setminus\mathrm{Im}\,a$ and $\ell<p-1$.
Then, the fiber $\psi^{-1}((L,z))$
decomposes as
\[\psi^{-1}((L,z))=\left(\psi^{-1}((L,z))\cap (\mathbb{P}(V_1)\times \mathcal{O}^G_{(n-\ell,\ell)^*})\right)\sqcup\left(\psi^{-1}((L,z))\cap (\mathbb{P}(V_1)\times \mathcal{O}^G_{(n-\ell-1,\ell+1)^*})\right) ; \]
both subsets in this union are irreducible; 
the first one is an affine space, and the second one is open and dense in $\psi^{-1}((L,z))$, and we have
\begin{eqnarray*}
& \dim \bigl( \psi^{-1}((L,z))\cap (\mathbb{P}(V_1)\times \mathcal{O}^G_{(n-\ell,\ell)^*}) \bigr) =\rk b,\\
& \dim \bigl( \psi^{-1}((L,z))\cap (\mathbb{P}(V_1)\times \mathcal{O}^G_{(n-\ell-1,\ell+1)^*}) \bigr) =p-\rk a-1.
\end{eqnarray*}
\end{proposition}

\begin{proof}
From the fact that every element $(L,x)\in\psi^{-1}((L,z))$ satisfies $x\in \overline{\mathcal{O}_P^G}$, and thus $x^2=0$
(see Section \ref{section-5-1}), and from the definition of $\Phi(z)$, we see that
\[\psi^{-1}((L,z))=\{(L,x):x\in\Phi(z)\ \ \mbox{and}\ \ \rk x\leq\min\{k,n-k\}\}.\]
Then the proposition is a consequence of Lemma \ref{lemma-3}.
\end{proof}

\begin{corollary}
\label{corollary-5-5}
Condition (\ref{20}) of Section \ref{section-5-2} is fulfilled.
\end{corollary}

\begin{proof}
Indeed, the map $\mathcal{O}_{\lambda}^G\cap(z+\mathfrak{n}_L)\to\psi^{-1}((L,z))\cap(\mathbb{P}(V_1)\times\mathcal{O}_{\lambda}^G)$, $x\mapsto (L,x)$ is an isomorphism of algebraic varieties. Hence the statement follows from Proposition \ref{proposition-4}.
\end{proof}

\subsection{Good pairs}
\label{section-5-4}
Let $((r,s),\mu)\in\Pi_2^k$, so $\ell:=r+s\leq\min\{p,q,k,n-k\}$.
By Proposition \ref{proposition-4}, we know that $\mathcal{Z}_{((r,s),\mu)}^\lambda=\emptyset$ unless \[\lambda=(n-\ell,\ell)^*\quad\mbox{or}\quad\big(\lambda=(n-\ell-1,\ell+1)^*\quad\mbox{and}\quad\ell<\min\{k,n-k\}\big).\]
It also follows from Proposition \ref{proposition-4} that the restriction
$\psi^\lambda:\mathcal{Z}_{((r,s),\mu)}^\lambda\to \mathfrak{O}_{((r,s),\mu)}^K$ of $\psi$
has a constant fiber. In the rest of the section, we fix an element
$(L,z)\in \mathfrak{O}_{((r,s),\mu)}^K$
with
\[z=\left(\begin{array}{cc}
0 & a \\ b & 0
\end{array}\right)
\quad\mbox{and}\quad
L=\langle v\rangle_\mathbb{C}\]
where $a\in M_{p,q}(\mathbb{C})$ is of rank $r$, $b\in M_{q,p}(\mathbb{C})$ is of rank $s$, and $v$ belongs to $\mathrm{Im}\,a$, $\ker b\setminus\mathrm{Im}\,a$, or $V_1\setminus\ker b$, depending on whether $((r,s),\mu)$ is of type (I), (II), or (III) (see Proposition \ref{proposition-2-5})
and we compute
the number
\begin{eqnarray*}
\delta_{((r,s),\mu)}^\lambda & := & \dim\mathcal{Z}^\lambda_{((r,s),\mu)}+\dim\mathcal{F}_{\lambda,(k,n-k)} \\
 & = & \dim\mathfrak{O}_{((r,s),\mu)}^K+\dim(\psi^\lambda)^{-1}((L,z))+\dim\mathcal{F}_{\lambda,(k,n-k)}.
\end{eqnarray*}
Note that
\begin{equation}
\label{26}
\mbox{the pair $(\mathfrak{O}_{((r,s),\mu)}^K,\lambda)$ is good if and only if $\delta_{((r,s),\mu)}^\lambda=\dim\mathcal{Y}$}
\end{equation}
(see Definition \ref{definition-good-pair}).
Note also that
\begin{equation}
\label{27-new}
\dim\mathcal{Y}=\dim \Gr_k(V)\times\mathbb{P}(V_1)=k(n-k)+p-1.
\end{equation}

\noindent
{\it Case 1: Assume $((r,s),\mu)$ of type (I).}

Thus Proposition \ref{proposition-2-5}\,{\rm (c)} yields $\dim\mathfrak{O}_{((r,s),\mu)}^K=\ell(n-\ell)+r-1$.
Moreover according to Proposition \ref{proposition-2-5}\,{\rm (b)} we have $v\in\mathrm{Im}\,a$. By Proposition \ref{proposition-4}\,{\rm (a)}, 
$\mathcal{Z}_{((r,s),\mu)}^{(n-\ell-1,\ell+1)^*}=\emptyset$, whereas
\[\delta_{((r,s),\mu)}^{(n-\ell,\ell)^*}=\big(\ell(n-\ell)+r-1\big)+\big(p-r\big)+(k-\ell)(n-k-\ell)=k(n-k)+p-1=\dim\mathcal{Y}\]
(using Propositions \ref{proposition-2}, \ref{proposition-4}, and relation (\ref{27-new})). We conclude (from (\ref{26})) that, for $((r,s),\mu)$ of type (I), the pair $(\mathfrak{O}_{((r,s),\mu)}^K,\lambda)$
 is good if and only if $\lambda=(n-\ell,\ell)^*$.
 
 \medskip
 \noindent
 {\it Case 2: Assume $((r,s),\mu)$ of type (III).}
 
 By Proposition \ref{proposition-2-5}\,{\rm (c)}, we have $\dim\mathfrak{O}_{((r,s),\mu)}^K=\ell(n-\ell)+p-1$. In addition Proposition \ref{proposition-2-5}\,{\rm (b)} yields $v\notin\ker b$. By Proposition \ref{proposition-4}\,{\rm (a)}, 
we have $\psi^{-1}((L,z))=\{(L,z)\}$, hence $\mathcal{Z}_{((r,s),\mu)}^{(n-\ell-1,\ell+1)^*}=\emptyset$ and
\[\delta_{((r,s),\mu)}^{(n-\ell,\ell)^*}=\big(\ell(n-\ell)+p-1\big)+0+(k-\ell)(n-k-\ell)=k(n-k)+p-1=\dim\mathcal{Y}.\]
Therefore, if $((r,s),\mu)$ is of type (III), 
the only good pair is $(\mathfrak{O}_{((r,s),\mu)}^K,(n-\ell,\ell)^*)$.

\medskip
\noindent
{\it Case 3: Assume $((r,s),\mu)$ of type (II).}

Proposition \ref{proposition-2-5}\,{\rm (c)} implies $\dim\mathfrak{O}_{((r,s),\mu)}^K=\ell(n-\ell)+p-s-1$.
In addition, by Proposition \ref{proposition-2-5}\,{\rm (b)} we have $v\in \ker b\setminus \mathrm{Im}\,a$.
By Proposition  \ref{proposition-4}\,{\rm (a)} and {\rm (b)} we see that $\dim (\psi^{(n-\ell,\ell)^*})^{-1}((L,z))=\rk b=s$. Hence
\[\delta_{((r,s),\mu)}^{(n-\ell,\ell)^*}=\big(\ell(n-\ell)+p-s-1\big)+s+(k-\ell)(n-k-\ell)=\dim\mathcal{Y}.\]
Therefore, the pair $(\mathfrak{O}_{((r,s),\mu)}^K,(n-\ell,\ell)^*)$ is good.

It remains to determine the nature of the pair $(\mathfrak{O}_{((r,s),\mu)}^K,\lambda)$ for $\lambda=(n-\ell-1,\ell+1)^*$.

First we note that if $\ell=\min\{k,n-k\}$ or $\ell\geq p-1$ (which may hold only if $((r,s),\mu)$ is of type $\mbox{(II)}^0$) 
then $\mathcal{Z}_{((r,s),\mu)}^\lambda=\emptyset$ (by Proposition \ref{proposition-4}\,{\rm (a)}). Thus the pair $(\mathfrak{O}_{((r,s),\mu)}^K,\lambda)$ is not good in that case.

Hereafter we assume that $\ell<\min\{k,n-k,p-1\}$. Invoking Proposition \ref{proposition-4}\,{\rm (b)} we see that
\begin{eqnarray*}
\delta_{((r,s),\mu)}^\lambda & = & \big(\ell(n-\ell)+p-s-1\big)+\big(p-r-1\big)+(k-\ell-1)(n-k-\ell-1) \\
 & = & 2p+\ell-n-1+k(n-k) \\
 & = & \dim\mathcal{Y}+(\ell-q).
\end{eqnarray*}
If $((r,s),\mu)$ is of type $\mbox{(II)}^0$, then $\ell<q$ (by Definition \ref{definition-2-9}, since the relations $\ell\leq p-2$ and $\ell<\min\{k,n-k\}$ are already satisfied): the pair $(\mathfrak{O}_{((r,s),\mu)}^K,\lambda)$ is not good in that case.
If $((r,s),\mu)$ is of type $\mbox{(II)}^*$, then we have in particular $\ell=q$, 
hence $\delta_{((r,s),\mu)}^\lambda = \dim \mathcal{Y}$: 
the pair $(\mathfrak{O}_{((r,s),\mu)}^K,\lambda)$ is good in that case.

\medskip
Combining Cases 1--3 we have shown:

\begin{proposition}
\label{proposition-5-6}
Let $ ((r, s), \mu) \in \Pi_2^k $ and set $ \ell = r + s $.  

\noindent
{\rm (a)} If $((r,s),\mu)$ is of type (I), $\mbox{(II)}^0$, or (III), then the pair $(\mathfrak{O}_{((r,s),\mu)}^K,\lambda)$ is good if and only if $\lambda=(n-\ell,\ell)^*$. \\
{\rm (b)} If $((r,s),\mu)$ is of type $\mbox{(II)}^*$, then the pair $(\mathfrak{O}_{((r,s),\mu)}^K,\lambda)$ is good if and only if $\lambda=(n-\ell,\ell)^*$ or $\lambda=(n-\ell-1,\ell+1)^*$.
\end{proposition}

Relations (\ref{23}) and (\ref{24}) are consequences of Proposition \ref{proposition-5-6}. As explained in Section \ref{section-5-2}, the proof of Theorem \ref{theorem-2} is now complete.

\section{On the combinatorial correspondence}

In this section, we show the results presented in Section \ref{section-2-3}.

\label{section-6}

\subsection{Proof of Proposition \ref{proposition-2-16}}\label{section-6-1}

We need to show the following three facts:
\begin{eqnarray}
\label{28} & \mbox{Every $(W,L)\in\Gr_k(V)\times\mathbb{P}(V_1)$ belongs to $\mathbb{O}_{(\tau,i)}$ for a unique $(\tau,i)\in\Theta_2^k$;} \\
\label{29} & \mbox{$\mathbb{O}_{(\tau,i)}$ is $K$-stable;} \\
\label{30} & \mbox{$\mathbb{O}_{(\tau,i)}$ is $K$-homogeneous.}
\end{eqnarray}

Condition (\ref{29}) easily follows from the definition of $\mathbb{O}_{(\tau,i)}$ and from the fact that $V_1$ and $V_2$ are stable by $K=\GL(V_1)\times\GL(V_2)$.

Let us check (\ref{28}). So let $(W,L)\in\Gr_k(V)\times\mathbb{P}(V_1)$. We set
\[n_1=\dim W\cap V_1,\qquad n_2=\dim W\cap V_2\quad\mbox{and}\quad \ell=k-n_1-n_2.\]
Clearly $n_1\leq p$, $n_2\leq q$, and $\ell\geq 0$ (since $n_1+n_2\leq \dim W=k$).
The natural inclusion 
$ W + V_2 \hookrightarrow V_1 \oplus V_2 $ induces an injection 
$ W / W \cap V_2 \hookrightarrow V_1 $, 
which implies $ k \leq p + n_2 $, whence $p-n_1-\ell\geq 0$.  
Similarly we get $q-n_2-\ell\geq 0$. 
Altogether these relations allow us to consider the $(1,2)$-tableau $\tau$ of shape $(p,q)$ and weight $k$, 
whose first column comprises 
$n_1$ entries $2$, $\ell$ entries $1$, and $ (p-n_1-\ell) $ entries $0$, 
and whose second column comprises 
$n_2$ entries $2$, $\ell$ entries $1$, and $ (q-n_2-\ell) $ entries $0$.
Set
\[i=\left\{\begin{array}{ll}
2 & \mbox{if $L\subset W$,} \\
1 & \mbox{if $L\subset W+V_2,\ L\not\subset W$,} \\
0 & \mbox{if $L\not\subset W+V_2$.}
\end{array}\right.\]
We claim that 
\begin{equation}
\label{31}
\mbox{the pair $(\tau,i)$ is an element of the set $\Theta_2^k$.}
\end{equation} 
If this is so, then the construction of $\tau$, $i$ and the definition of $\mathbb{O}_{(\tau,i)}$ guarantee that $(W,L)\in\mathbb{O}_{(\tau,i)}$, and that the pair $(\tau,i)$ is unique for this property. Hence (\ref{31}) is sufficient for completing the justification of (\ref{28}). For showing (\ref{31}), we just need to check that the number $i$ appears in the first column of $\tau$. We distinguish three cases.

\medskip
\noindent
{\it Case 1: $L\subset W$, that is, $i=2$.} 

In this case $L\subset W\cap V_1$, hence $W\cap V_1\not=0$. This yields $n_1\geq 1$, thus there is at least one label $2$ in the first column of $\tau$.

\medskip
\noindent
{\it Case 2: $L\subset W+V_2$ and $L\not\subset W$, that is, $i=1$.}

Then $(W+V_2)\cap V_1\not= W\cap V_1$. This implies that $W\not=(W\cap V_1)\oplus(W\cap V_2)$, whence $k>n_1+n_2$, and so $\ell\geq 1$. The latter inequality means that there is at least one label $1$ in the first column of $\tau$.

\medskip
\noindent
{\it Case 3: $L\not\subset W+V_2$, that is, $i=0$.}

We then have $V\not=W+V_2$, thus $p+q>\dim (W+V_2)=k+q-n_2$, so $p>k-n_2=\ell+n_1$. 
Thereby $p-n_1-\ell\geq 1$, which implies that the first column of $\tau$ contains at least one label $0$.

\medskip
In each case we conclude that $i$ appears in the first column of $\tau$, whence (\ref{31}) is valid. This completes the verification of (\ref{28}).

It remains to show (\ref{30}). Let $(W,L)\in\mathbb{O}_{(\tau,i)}$.
Let $n_1$ (resp., $n_2$) be the number of $2$'s in the first (resp., second) column of $\tau$ and set $\ell=k-n_1-n_2$, which is therefore the number of $1$'s in the first and in the second column of $\tau$.
Choose a basis $(e_1,\ldots,e_{n_1})$ of $W\cap V_1$ and a basis $(f_1,\ldots,f_{n_2})$ of $W\cap V_2$. Choose vectors $v_1,\ldots,v_\ell\in W$ which complete $(e_1,\ldots,e_{n_1},f_1,\ldots,f_{n_2})$ into a basis of $W$. For every $j\in\{1,\ldots,\ell\}$ we have $v_j=e_{n_1+j}+f_{n_2+j}$ for some $e_{n_1+j}\in V_1$, $f_{n_2+j}\in V_2$, and the vectors $e_1,\ldots,e_{n_1+\ell},f_1,\ldots,f_{n_2+\ell}$ are linearly independent. Choose vectors $e_{n_1+\ell+1},\ldots,e_p\in V_1$ and $f_{n_2+\ell+1},\ldots,f_q\in V_2$ which complete $(e_1,\ldots,e_{n_1+\ell})$ and $(f_1,\ldots,f_{n_2+\ell})$ into bases of $V_1$ and $V_2$.

Let $v\in V$ such that $L=\langle v\rangle_\mathbb{C}$.
In the case where $i=2$, we have $L\subset W$, hence we can choose $e_1=v$.
In the case where $i=1$, i.e., $L\subset W+V_2$ and $L\not\subset W$, we can choose $e_{n_1+1}=v$. Finally if $i=0$, i.e., $L\not\subset W+V_2$, we can take $e_p=v$.

Similarly, for a second element $(W',L')$ of $\mathbb{O}_{(\tau,i)}$, we can construct bases $(e'_1,\ldots,e'_p)$ and $(f'_1,\ldots,f'_q)$ of $V_1$ and $V_2$, respectively, such that 
\[W'=\langle e'_1,\ldots,e'_{n_1},f'_1,\ldots,f'_{n_2},e'_{n_1+1}+f'_{n_2+1},\ldots,e'_{n_1+\ell}+f'_{n_2+\ell}\rangle_\mathbb{C}\]
and $L'=\langle e'_1\rangle_\mathbb{C}$ if $i=2$, 
$L'=\langle e'_{n_1+1}\rangle_\mathbb{C}$ if $i=1$, and $L'=\langle e'_{p}\rangle_\mathbb{C}$ if $i=0$. Let $g\in K=\GL(V_1)\times\GL(V_2)$ be the automorphism such that $g(e_j)=e'_j$ for all $j\in\{1,\ldots,p\}$ and $g(f_j)=f'_j$ for all $j\in\{1,\ldots,q\}$. By construction we get $(W',L')=g(W,L)$. This shows (\ref{30}). The proof of Proposition \ref{proposition-2-16} is now complete.

\subsection{Proofs of Examples \ref{example-2-17} and \ref{example-2-18}}

\label{section-6-2}

Before discussing in detail the special situations of Examples \ref{example-2-17} and \ref{example-2-18},
let us describe the general technique for determining the surjection $\Phi:\Xfv/K\to \mathfrak{E}/K$ of Corollaries \ref{corollary-1} and \ref{corollary-2-11}:
\begin{itemize}
\item Let a $K$-orbit $\mathfrak{O}\subset\mathfrak{E}$ and let us determine $\Phi^{-1}(\mathfrak{O})$.
\item
Take any representative $(L,z)\in\mathfrak{O}(\subset\mathbb{P}(V_1)\times\mathcal{N}(\mathfrak{s})_2^k)$.
\item
Let a partition $\lambda\vdash n$ such that the pair $(\mathfrak{O},\lambda)$ is good (in the sense of Definition \ref{definition-good-pair} or according to the explicit description of Proposition \ref{proposition-5-6}).
\item
Take a generic element $x\in\mathcal{O}^G_\lambda$ such that $\mathrm{Im}\,x^\theta\subset L\subset\ker x^\theta$ and $x^{-\theta}=z$.
\item
Take a generic subspace $W\in\Gr_k(V)$ such that $\mathrm{Im}\,x\subset W\subset\ker x$.
\item Let $\mathbb{O}\subset\mathfrak{X}$ be the $K$-orbit of the pair $(W,L)\in\mathfrak{X}(=\Gr_k(V)\times\mathbb{P}(V_1))$, then we have $\Phi(\mathbb{O})=\mathfrak{O}$.
\item If $\lambda$ is the unique partition such that the pair $(\mathfrak{O},\lambda)$ is good, then $\Phi^{-1}(\mathfrak{O})=\{\mathbb{O}\}$. If there is a second partition $\lambda'$ such that $(\mathfrak{O},\lambda')$ is good, then arguing similarly with $\lambda'$ in place of $\lambda$ we find a second $K$-orbit $\mathbb{O}'\subset\Xfv$ such that $\Phi(\mathbb{O}')=\mathfrak{O}$, and we have $\Phi^{-1}(\mathfrak{O})=\{\mathbb{O},\mathbb{O}'\}$.
\end{itemize}
The justification is as follows. By definition of the map $\Phi$ in Corollary \ref{corollary-1}, we have $\Phi(\mathbb{O})=\mathfrak{O}$ if and only if the inclusion $T^*_\mathbb{O}\Xfv\subset\overline{\pi^{-1}(\mathfrak{O})}$ holds. Recall from Section \ref{section-1-3} that the closure of the conormal bundle $T^*_\mathbb{O}\Xfv$ is an irreducible component of the conormal variety $\mathcal{Y}$. On the other hand, we know from Sections \ref{section-5-1}--\ref{section-5-2} that the components of $\mathcal{Y}$ contained in $\overline{\pi^{-1}(\mathfrak{O})}$ correspond
to the good pairs of the form $(\mathfrak{O},\lambda)$. Specifically if $(\mathfrak{O},\lambda)$ is a good pair then the set
\[K\{(W,L,x):(W,x)\in\Gr_k(V)\times \mathcal{O}_{\lambda}^G \mbox{ satisfying (\ref{32})}\},\]
where
\begin{equation}
\label{32} \mathrm{Im}\,x^\theta\subset L\subset\ker x^\theta,\ \mathrm{Im}\,x\subset W\subset \ker x,\ \mbox{ and }\ x^{-\theta}=z,
\end{equation}
is dense in a component $\mathcal{C}_{(\mathfrak{O},\lambda)}\subset\mathcal{Y}$ contained in $\overline{\pi^{-1}(\mathfrak{O})}$, and every such component is of this form. Therefore $\mathcal{C}_{(\mathfrak{O},\lambda)}=\overline{T^*_\mathbb{O}\Xfv}$ if and only if $(W,L,x)\in T^*_\mathbb{O}\Xfv$ for any generic pair $(W,x)\in\Gr_k(V)\times\mathcal{O}_{\lambda}^G$ satisfying (\ref{32}), which is equivalent to having $(W,L)\in\mathbb{O}$.

\begin{proof}[Proof of Example \ref{example-2-17}]
Here we have $p=3$, $q=1$, and $k\in\{1,2\}$.

In Tables \ref{table-1} and \ref{table-2} we calculate the correspondence $\phi:\Theta_2^k\to\Pi_2^k$ for $k=1$ and $k=2$, by using the technique described above.
In the first column of each table we enumerate the various elements of $\Pi_2^k$.
In the second column we choose a representative $(L,z)$ of the corresponding orbit $\mathfrak{O}^K_{((r,s),\mu)}$.
In the third column we indicate the partition(s) $\lambda$ such that the pair $(\mathfrak{O}_{((r,s),\mu)}^K,\lambda)$ is good (see Proposition \ref{proposition-5-6}). 
In the fourth column (with the help of Lemma \ref{lemma-3}) and in the  fifth column we compute the pairs $(W,x)\in \Gr_k(V)\times\mathcal{N}(\mathfrak{g})$ of elements satisfying (\ref{32}) (possibly depending on scalars $\alpha,\beta$).
In the sixth column we indicate the values of $n_j=\dim W\cap V_j$ (for $j\in\{1,2\}$) and the relative position of $L$ and $W$ when the pair $(W,x)$ is generic. In the last column, taking Proposition \ref{proposition-2-16} into account, we deduce
the $(1,2)$-tableau $(\tau,i)$ such that the latter pair $(W,L)$ belongs to $\mathbb{O}_{(\tau,i)}$. In each case we recover the correspondence $\phi$ stated in Example \ref{example-2-17}, and this observation completes the proof.
\end{proof}

\begin{table}[hbtp]
\tiny
\caption{Proof of Example \ref{example-2-17} in the case $k=1$}
\label{table-1}
\renewcommand{\arraystretch}{3}
\begin{tabular}{c||c|c|c|c|c|c|}
$((r,s),\mu)$ & $(L,z)$ & $\lambda$ & $x$ & $W$ & $\renewcommand{\arraystretch}{1}\begin{array}{c}(n_1,n_2) \\ (W,L) \end{array}$ & $(\tau,i)$ \\
\hline\hline
$\verysmallyoung{+\sbarm,::+,::+}$ & $\smallmatrixt{* \\ 0 \\ 0 \\ \hline 0},\smallmatrixt{0 & 0 & 0  & 1 \\ 0 & 0 & 0 & 0 \\ 0 & 0 & 0 & 0 \\ \hline 0 & 0 & 0 & 0}$ & $(2,1^2)$ & $\renewcommand{\arraystretch}{1.1}\begin{array}[t]{c}\smallmatrixt{0 & \alpha & \beta  & 1 \\ 0 & 0 & 0 & 0 \\ 0 & 0 & 0 & 0 \\ \hline 0 & 0 & 0 & 0}\\ \alpha,\beta\in\mathbb{C}\end{array}$ & $\smallmatrixt{* \\ 0 \\ 0 \\ \hline 0}$ & $\renewcommand{\arraystretch}{1.2}\begin{array}{c} (1,0) \\ L=W\end{array}$ & $\!\left(\verysmallyoung{00,0,2},2\right)\!$ \\
\hline
\raisebox{0pt}[6ex][5ex]{}
$\verysmallyoung{+\sbarm,\spbar,\spbar}$ & $\smallmatrixt{0 \\ * \\ 0 \\ \hline 0},\smallmatrixt{0 & 0 & 0  & 1 \\ 0 & 0 & 0 & 0 \\ 0 & 0 & 0 & 0 \\ \hline 0 & 0 & 0 & 0}$ & $(2,1^2)$ & $\smallmatrixt{0 & 0 & 0 & 1 \\ 0 & 0 & 0 & 0 \\ 0 & 0 & 0 & 0 \\ \hline 0 & 0 & 0 & 0}$ & $\smallmatrixt{* \\ 0 \\ 0 \\ \hline 0}$ & $\renewcommand{\arraystretch}{1.2}\begin{array}{c} (1,0) \\ L\not\subset W{+}V_2\end{array}$ & $\!\left(\verysmallyoung{00,0,2},0\right)\!$ \\
\hline
$\verysmallyoung{:\sbarm+,\spbar,\spbar}$ & $\smallmatrixt{0 \\ * \\ 0 \\ \hline 0},\smallmatrixt{0 & 0 & 0  & 0 \\ 0 & 0 & 0 & 0 \\ 0 & 0 & 0 & 0 \\ \hline 1 & 0 & 0 & 0}$ & $(2,1^2)$ & $\renewcommand{\arraystretch}{1.1}\begin{array}[t]{c}\smallmatrixt{0 & 0 & 0  & 0 \\ \alpha & 0 & 0 & 0 \\ 0 & 0 & 0 & 0 \\ \hline 1 & 0 & 0 & 0}\\ \alpha\in\mathbb{C}\end{array}$ & $\langle\smallmatrixt{0 \\ \alpha \\ 0 \\ \hline 1}\rangle_\mathbb{C}$ & $\renewcommand{\arraystretch}{1.2}\begin{array}{c} (0,0) \\ L{\subset}W{+}V_2,\,L{\not\subset} W\end{array}$ & $\!\left(\verysmallyoung{01,0,1},1\right)\!$ \\
\hline
\raisebox{0pt}[6ex][5ex]{}
$\verysmallyoung{-\spbar,:\spbar,:\spbar}$ & $\smallmatrixt{* \\ 0 \\ 0 \\ \hline 0},\smallmatrixt{0 & 0 & 0  & 0 \\ 0 & 0 & 0 & 0 \\ 0 & 0 & 0 & 0 \\ \hline 1 & 0 & 0 & 0}$ & $(2,1^2)$ & $\smallmatrixt{0 & 0 & 0  & 0 \\ 0 & 0 & 0 & 0 \\ 0 & 0 & 0 & 0 \\ \hline 1 & 0 & 0 & 0}$ & $V_2$ & $\renewcommand{\arraystretch}{1.2}\begin{array}{c} (0,1) \\ L\not\subset W{+}V_2\end{array}$ & $\!\left(\verysmallyoung{02,0,0},0\right)\!$ \\
\hline
\raisebox{0pt}[8ex][6ex]{}
$\verysmallyoung{\spbar,\spbar,\spbar,:\sbarm}$ & $\smallmatrixt{* \\ 0 \\ 0 \\ \hline 0},\smallmatrixt{0 & 0 & 0  & 0 \\ 0 & 0 & 0 & 0 \\ 0 & 0 & 0 & 0 \\ \hline 0 & 0 & 0 & 0}$ & $(1^4)$ & $\smallmatrixt{0 & 0 & 0  & 0 \\ 0 & 0 & 0 & 0 \\ 0 & 0 & 0 & 0 \\ \hline 0 & 0 & 0 & 0}$ & any line & $\renewcommand{\arraystretch}{1.2}\begin{array}{c} (0,0) \\ L\not\subset W{+}V_2\end{array}$ & $\!\left(\verysmallyoung{01,0,1},0\right)\!$ \\
\hline
\end{tabular}
\renewcommand{\arraystretch}{1}
\end{table}

\begin{table}[hbtp]
\tiny
\caption{Proof of Example \ref{example-2-17} in the case $k=2$}
\label{table-2}
\renewcommand{\arraystretch}{2}
\begin{tabular}{c||c|c|c|c|c|c|}
$((r,s),\mu)$ & $(L,z)$ & $\lambda$ & $x$ & $W$ & $\renewcommand{\arraystretch}{1}\begin{array}{c}(n_1,n_2) \\ (W,L) \end{array}$ & $(\tau,i)$ \\
\hline\hline
\raisebox{0pt}[6ex][5ex]{}
$\verysmallyoung{+\sbarm,::+,::+}$ & $\smallmatrixt{* \\ 0 \\ 0 \\ \hline 0},\smallmatrixt{0 & 0 & 0  & 1 \\ 0 & 0 & 0 & 0 \\ 0 & 0 & 0 & 0 \\ \hline 0 & 0 & 0 & 0}$ & $(2,1^2)$ & $\renewcommand{\arraystretch}{1.1}\begin{array}[t]{c}\smallmatrixt{0 & \alpha & \beta  & 1 \\ 0 & 0 & 0 & 0 \\ 0 & 0 & 0 & 0 \\ \hline 0 & 0 & 0 & 0}\\ \alpha,\beta\in\mathbb{C}\end{array}$ & $\smallmatrixt{* \\ 0 \\ 0 \\ \hline 0}\subset W\subset\ker x$ & $\renewcommand{\arraystretch}{1.2}\begin{array}{c} (1,0) \\ L\subset W\end{array}$ & $\!\left(\verysmallyoung{01,1,2},2\right)\!$ \\
\hline
\raisebox{0pt}[6ex][5ex]{}
$\verysmallyoung{+\sbarm,\spbar,\spbar}$ & $\smallmatrixt{0 \\ * \\ 0 \\ \hline 0},\smallmatrixt{0 & 0 & 0  & 1 \\ 0 & 0 & 0 & 0 \\ 0 & 0 & 0 & 0 \\ \hline 0 & 0 & 0 & 0}$ & $(2,1^2)$ & $\smallmatrixt{0 & 0 & 0 & 1 \\ 0 & 0 & 0 & 0 \\ 0 & 0 & 0 & 0 \\ \hline 0 & 0 & 0 & 0}$ & $\smallmatrixt{* \\ 0 \\ 0 \\ \hline 0}\subset W\subset V_1$ & $\renewcommand{\arraystretch}{1.2}\begin{array}{c} (2,0) \\ L\not\subset W{+}V_2\end{array}$ & $\!\left(\verysmallyoung{00,2,2},0\right)\!$ \\
\cline{3-7}
 & &$(2,2)$ & $\renewcommand{\arraystretch}{1.1}\begin{array}[t]{c}\smallmatrixt{0 & 0 & 0 & 1 \\ 0 & 0 & \alpha & 0 \\ 0 & 0 & 0 & 0 \\ \hline 0 & 0 & 0 & 0}\\ \alpha\not=0\end{array}$ & $\smallmatrixt{* \\ * \\ 0 \\ \hline 0}$ & $\renewcommand{\arraystretch}{1.2}\begin{array}{c} (2,0) \\ L\subset W\end{array}$ & $\!\left(\verysmallyoung{00,2,2},2\right)\!$ \\
\hline
\raisebox{0pt}[6ex][5ex]{}
$\verysmallyoung{:\sbarm+,\spbar,\spbar}$ & $\smallmatrixt{0 \\ * \\ 0 \\ \hline 0},\smallmatrixt{0 & 0 & 0  & 0 \\ 0 & 0 & 0 & 0 \\ 0 & 0 & 0 & 0 \\ \hline 1 & 0 & 0 & 0}$ & $(2,1^2)$ & $\renewcommand{\arraystretch}{1.1}\begin{array}[t]{c}\smallmatrixt{0 & 0 & 0  & 0 \\ \alpha & 0 & 0 & 0 \\ 0 & 0 & 0 & 0 \\ \hline 1 & 0 & 0 & 0}\\ \alpha\in\mathbb{C}\end{array}$ & $\langle\smallmatrixt{0 \\ \alpha \\ 0 \\ \hline 1}\rangle_\mathbb{C}\subset W\subset\smallmatrixt{0 \\ * \\ * \\ \hline *}$ & $\renewcommand{\arraystretch}{1.2}\begin{array}{c} (1,0) \\ L{\subset} W{+}V_2,\,L{\not\subset} W\end{array}$ & $\!\left(\verysmallyoung{01,1,2},1\right)\!$ \\
\cline{3-7}
 & & $(2,2)$ & $\renewcommand{\arraystretch}{1.1}\begin{array}[t]{c} \smallmatrixt{0 & 0 & 0  & 0 \\ \alpha & 0 & \beta & 0 \\ 0 & 0 & 0 & 0 \\ \hline 1 & 0 & 0 & 0}\\ \alpha\in\mathbb{C},\beta\not=0\end{array}$ & $\smallmatrixt{0 \\ * \\ 0 \\ \hline *}$ & $\renewcommand{\arraystretch}{1.2}\begin{array}{c} (1,1) \\ L\subset W\end{array}$ & $\!\left(\verysmallyoung{02,0,2},2\right)\!$ \\
\hline
\raisebox{0pt}[6ex][5ex]{}
$\verysmallyoung{-\spbar,:\spbar,:\spbar}$ & $\smallmatrixt{* \\ 0 \\ 0 \\ \hline 0},\smallmatrixt{0 & 0 & 0  & 0 \\ 0 & 0 & 0 & 0 \\ 0 & 0 & 0 & 0 \\ \hline 1 & 0 & 0 & 0}$ & $(2,1^2)$ & $\smallmatrixt{0 & 0 & 0  & 0 \\ 0 & 0 & 0 & 0 \\ 0 & 0 & 0 & 0 \\ \hline 1 & 0 & 0 & 0}$ & $V_2\subset W\subset\smallmatrixt{0 \\ * \\ * \\ \hline *}$ & $\renewcommand{\arraystretch}{1.2}\begin{array}{c} (1,1) \\ L\not\subset W{+}V_2\end{array}$ & $\!\left(\verysmallyoung{02,0,2},0\right)\!$ \\
\hline
\raisebox{0pt}[8ex][6ex]{}
$\verysmallyoung{\spbar,\spbar,\spbar,:\sbarm}$ & $\smallmatrixt{* \\ 0 \\ 0 \\ \hline 0},\smallmatrixt{0 & 0 & 0  & 0 \\ 0 & 0 & 0 & 0 \\ 0 & 0 & 0 & 0 \\ \hline 0 & 0 & 0 & 0}$ & $(1^4)$ & $\smallmatrixt{0 & 0 & 0  & 0 \\ 0 & 0 & 0 & 0 \\ 0 & 0 & 0 & 0 \\ \hline 0 & 0 & 0 & 0}$ & any & $\renewcommand{\arraystretch}{1.2}\begin{array}{c} (1,0) \\ L\not\subset W{+}V_2\end{array}$ & $\!\left(\verysmallyoung{01,1,2},0\right)\!$ \\
\hline
\end{tabular}
\renewcommand{\arraystretch}{1}
\end{table}

\begin{proof}[Proof of Example \ref{example-2-18}]
Here we suppose $p=q=k=\frac{n}{2}=:m$.
We consider an element $((r,s),\mu)\in\Pi_2^k$ and 
we compute the element $(\tau,i):=\phi^{-1}(((r,s),\mu))\in\Theta_2^k$
with the technique described at the beginning of this subsection.
Let $\ell=r+s$.

We fix a representative $(L,z)\in\mathfrak{O}_{((r,s),\mu)}^K$ and we write $L=\langle v\rangle_\mathbb{C}$ with $v\in V_1\setminus\{0\}$ and
$z=\left(\begin{smallmatrix}
0 & a \\ b & 0
\end{smallmatrix}\right)$ with $a\in M_m(\mathbb{C})$ and $b\in M_m(\mathbb{C})$, $\rk a=r$, $\rk b=s$.  By Proposition \ref{proposition-5-6}, $\lambda:=(n-\ell,\ell)^*$ is the only partition such that the pair $(\mathfrak{O}_{((r,s),\mu)}^K,\lambda)$ is good.

First assume $((r,s),\mu)$ of type (I). Thus $L\subset\mathrm{Im}\,a\subset\ker b$.
By Lemma \ref{lemma-3}, the elements $x\in\mathcal{O}_{\lambda}^G$ such that \begin{equation}
\label{33}
\mathrm{Im}\,x^\theta\subset L\subset\ker x^\theta\quad\mbox{and}\quad x^{-\theta}=z\end{equation} are of the form
\[x=\left(\begin{matrix}
v\cdot{}^tu & a \\ b & 0
\end{matrix}\right)\quad\mbox{with $u\in\ker{}^ta$}.\]
Since $v\in\mathrm{Im}\,x$, we get $\mathrm{Im}\,x=\mathrm{Im}\,a\oplus\mathrm{Im}\,b$. Hence every subspace $W\in\Gr_k(V)$ such that $\mathrm{Im}\,x\subset W\subset \ker x$ satisfies $\dim W\cap V_1\geq\dim\mathrm{Im}\,a=r$, 
$\dim W\cap V_2\geq\dim\mathrm{Im}\,b=s$, and $L\subset \mathrm{Im}\,a\subset W$. In fact choosing a basis $(v_1,\ldots,v_{m-\ell})$ (resp., $(w_1,\ldots,w_{m-\ell})$) of some supplementary subspace of $\mathrm{Im}\,a$ in $\ker b$ (resp., of $\mathrm{Im}\,b$ in $\ker a$) and letting $w\in V_2$ such that $v=aw$, it turns out that the subspace
\[
W:=\mathrm{Im}\,a\oplus\mathrm{Im}\,b\oplus\langle v_j+w_j-({}^tu\cdot v_j)w:1\leq j\leq m-\ell\rangle_\mathbb{C}\in\Gr_k(V)
\]
satisfies $\mathrm{Im}\,x\subset W\subset\ker x$ and realizes the equalities $\dim W\cap V_1=r$ and $\dim W\cap V_2=s$.
Therefore generically those equalities hold. Hence the pair $(\tau,i)$ is such that the $(1,2)$-tableau $\tau$ contains $r$ entries $2$ in its first column and $s$ entries $2$ in its second column, and $i=2$ (since $L\subset W$), as claimed in Example \ref{example-2-18}.

Next assume $((r,s),\mu)$ of type (III), that is, $L\not\subset \ker b$.
In that case by Lemma \ref{lemma-3} the only element $x\in\mathcal{O}_{\lambda}^G$ satisfying (\ref{33}) is $x=z$.
Thus $\mathrm{Im}\,x=\mathrm{Im}\,a\oplus\mathrm{Im}\,b$ and $\ker x=\ker a\oplus\ker b$. Every $W\in\Gr_k(V)$ such that $\mathrm{Im}\,x\subset W\subset\ker x$ satisfies $\dim W\cap V_1\geq r$, $\dim W\cap V_2\geq s$, and $L\not\subset W+V_2$ (since $(W+V_2)\cap V_1\subset \ker b$ and $L\not\subset \ker b$). In addition the subspace 
\[W:=\mathrm{Im}\,a\oplus\mathrm{Im}\,b\oplus\langle v_j+w_j:1\leq j\leq m-\ell\rangle_\mathbb{C}\in\Gr_k(V)\] 
(with $v_1,\ldots,v_{m-\ell},w_1,\ldots,w_{m-\ell}$ as above)
realizes the equalities $\dim W\cap V_1=r$ and $\dim W\cap V_2=s$, therefore those equalities hold for a generic $W$.
This implies that the pair $(\tau,i)$ is such that $\tau$ is the $(1,2)$-tableau containing $r$ entries $2$ in its first column and $s$ entries $2$ in its second column, and $i=0$ (since $L\not\subset W+V_2$), which agrees with the statement of Example \ref{example-2-18}.

Finally assume $((r,s),\mu)$ of type (II), thus $v\in \ker b\setminus\mathrm{Im}\,a$. By Lemma \ref{lemma-3}, the elements $x\in\mathcal{O}_{\lambda}^G$ fulfilling (\ref{33}) are of the form
\[x=\left(\begin{matrix}
v\cdot{}^tu & a \\ b & 0
\end{matrix}\right)\quad\mbox{with $u\in\mathrm{Im}\,{}^tb$}.\]
We see that $\mathrm{Im}\,a\subset\mathrm{Im}\,x\subset\mathrm{Im}\,a\oplus(\mathrm{Im}\,b+L)$ and $\ker x=\ker a\oplus\ker b$.
Every $W\in\Gr_k(V)$ such that $\mathrm{Im}\,x\subset W\subset \ker x$ satisfies 
$\dim W\cap V_1\geq r$ (since $\mathrm{Im}\,a\subset W\cap V_1$) and
\begin{eqnarray}
\label{38-new}
\dim W\cap V_2 & = & \dim W+\dim V_2-\dim(W+V_2) \\ & \geq & \dim W+\dim V_2-\dim(\ker b\oplus V_2)=s.
\nonumber
\end{eqnarray}
Choose a basis $(v_1=v,v_2,\ldots,v_{m-\ell})$ (resp., $(w_1,\ldots,w_{m-\ell})$) of a supplementary subspace of $\mathrm{Im}\,a$ in $\ker b$ (resp., of $\mathrm{Im}\,b$ in $\ker a$) and set 
\[W=\mathrm{Im}\,x\oplus\langle  v_j+w_j:1\leq j\leq m-\ell\rangle_\mathbb{C}\in\Gr_k(V).\] 
Then 
$\mathrm{Im}\,x\subset W\subset\ker x$
and 
it is readily seen that this subspace $W$ satisfies
\begin{equation}
\label{36}
\dim W\cap V_1=r,\quad \dim W\cap V_2=s,\quad \mbox{and}\quad L\not\subset W.\end{equation}
Hence (\ref{36}) holds for a generic $W$. 
Therefore by (\ref{38-new}) a generic $W$ also satisfies the inclusion
\[L\subset\ker b\subset W+V_2.\]
We conclude that $(\tau,i)$ is such that the $(1,2)$-tableau $\tau$ has $r$ entries $2$ in its first column, $s$ entries $2$ in its second column, and $i=1$. This coincides with the statement of Example \ref{example-2-18}. The verification of Example \ref{example-2-18} is complete.
\end{proof}

\section*{Appendix}

In Table \ref{table-3} we give a further example of the orbit correspondence in the case of 
$ (G,K) = (\GL_6(\mathbb{C}), \GL_3(\mathbb{C}){\times}\GL_3(\mathbb{C})), \; \mathfrak{X} = \Gr_3(\mathbb{C}^6){\times}\mathbb{P}(\mathbb{C}^3) $.
In this case, the correspondence is a bijection (see Example~\ref{example-2-18}).


\bigskip

\begin{table}[hbtp]
\caption{Orbit correspondence for $ p = q = k = 3 $}
\label{table-3}
{\tiny
\renewcommand{\pbar}{\,{+}\hspace*{.4pt}\rule[-1.5pt]{1.4pt}{2.55ex}}
\renewcommand{\barm}{\hspace*{-.1pt}\rule[-2pt]{1.4pt}{2.55ex}\hspace*{.8pt}{-}\,}
\renewcommand{\spbar}{\,{+}\hspace*{.6pt}\rule[-2pt]{1.7pt}{2.65ex}}
\renewcommand{\sbarm}{\hspace*{-.1pt}\rule[-2pt]{1.4pt}{2.65ex}\hspace*{.8pt}{-}\,}

\noindent
$
\begin{array}{c|c|c|c|c|c|c|c|c|c}
\raisebox{0pt}[4ex][5ex]{}
(\tau, i) \in \Theta_2^3 & 
\!\left(\verysmallyoung{20,20,20}, 2\right)\!  & 
\!\left(\verysmallyoung{00,20,22}, 0\right)\!  & 
\!\left(\verysmallyoung{00,20,22}, 2\right)\!  & 
\!\left(\verysmallyoung{00,02,22}, 0\right)\!  & 
\!\left(\verysmallyoung{00,02,22}, 2\right)\!  & 
\!\left(\verysmallyoung{02,02,02}, 0\right)\!  
\\
\hline
\raisebox{0pt}[8ex][3ex]{}
\phi(\tau,i)\in \Pi_2^3 &
\verysmallyoung{+\barm,+\barm,+\barm} & 
\verysmallyoung{:+\barm,:+\barm,-\pbar} & 
\verysmallyoung{+\barm,+\barm,:\barm{+}} & 
\verysmallyoung{:+\barm,-\pbar,-\pbar} & 
\verysmallyoung{+\barm,:\barm{+},:\barm{+}} & 
\verysmallyoung{-\pbar,-\pbar,-\pbar}
\end{array}
$
\bigskip

\noindent
$
\begin{array}{c|c|c|c|c|c|c|c|c|c}
\raisebox{0pt}[4ex][5ex]{}
(\tau, i) \in \Theta_2^3 & 
\!\left(\verysmallyoung{10,20,21}, 1\right)\!  & 
\!\left(\verysmallyoung{10,20,21}, 2\right)\!  & 
\!\left(\verysmallyoung{00,11,22}, 0\right)\!  &
\!\left(\verysmallyoung{00,11,22}, 1\right)\!  & 
\!\left(\verysmallyoung{00,11,22}, 2\right)\!  & 
\!\left(\verysmallyoung{01,02,12}, 0\right)\! 
\\
\hline
\raisebox{0pt}[8ex][3ex]{}
\phi(\tau,i)\in \Pi_2^3 &
\verysmallyoung{+\barm,+\barm,\pbar,:\barm} & 
\verysmallyoung{+\barm,+\barm,::{+},:\barm} & 
\verysmallyoung{:+\barm,{-}\pbar,:\pbar,::\barm} &
\verysmallyoung{+\barm,:\barm{+},\pbar,:\barm} & 
\verysmallyoung{+\barm,:\barm{+},::{+},:\barm} & 
\verysmallyoung{-\pbar,-\pbar,:\pbar,::\barm}
\end{array}
$
\bigskip

\noindent
$
\begin{array}{c|c|c|c|c|c|c|c|c|c}
\raisebox{0pt}[4ex][5ex]{}
(\tau, i) \in \Theta_2^3 &
\!\left(\verysmallyoung{01,02,12}, 1\right)\!  & 
\!\left(\verysmallyoung{10,11,21}, 1\right)\!  & 
\!\left(\verysmallyoung{10,11,21}, 2\right)\!  & 
\!\left(\verysmallyoung{01,11,12}, 0\right)\!  & 
\!\left(\verysmallyoung{01,11,12}, 1\right)\!  & 
\!\left(\verysmallyoung{11,11,11}, 1\right)\! 
\\
\hline
\raisebox{0pt}[10ex][3ex]{}
\phi(\tau,i) \in \Pi_2^3 &
\verysmallyoung{:\barm{+},:\barm{+},\pbar,:\barm} & 
\verysmallyoung{+\barm,\pbar,\pbar,:\barm,:\barm} & 
\verysmallyoung{+\barm,::{+},::{+},:\barm,:\barm} & 
\verysmallyoung{{-}\pbar,:\pbar,:\pbar,::\barm,::\barm} & 
\verysmallyoung{:\barm{+},\pbar,\pbar,:\barm,:\barm} & 
\verysmallyoung{\pbar,\pbar,\pbar,:\barm,:\barm,:\barm} 
\end{array}
$
}
\end{table}

\bigskip

\noindent
\textbf{Acknowledgements.}\;\;
L.F. thanks Aoyama Gakuin University for warm hospitality and support in January 2014.   
K.N. is grateful to the generous support and hospitality of Universit\'e de Lorraine while he visited there in September 2015. This work was conceived during these visits. 
The authors thank the anonymous referee for pointing out some references and useful comments.

\renewcommand{\MR}[1]{}

\def\cftil#1{\ifmmode\setbox7\hbox{$\accent"5E#1$}\else
  \setbox7\hbox{\accent"5E#1}\penalty 10000\relax\fi\raise 1\ht7
  \hbox{\lower1.15ex\hbox to 1\wd7{\hss\accent"7E\hss}}\penalty 10000
  \hskip-1\wd7\penalty 10000\box7} \def\cprime{$'$} \def\cprime{$'$}
  \def\Dbar{\leavevmode\lower.6ex\hbox to 0pt{\hskip-.23ex \accent"16\hss}D}
\providecommand{\bysame}{\leavevmode\hbox to3em{\hrulefill}\thinspace}
\providecommand{\MR}{\relax\ifhmode\unskip\space\fi MR }
\providecommand{\MRhref}[2]{%
  \href{http://www.ams.org/mathscinet-getitem?mr=#1}{#2}
}
\providecommand{\href}[2]{#2}

\end{document}